\documentclass[12pt,a4paper,reqno]{amsart}

\usepackage{a4wide}
\usepackage{amssymb,amsmath,amsthm}
\usepackage{graphicx}
\usepackage{float}
\usepackage{colortbl}
\usepackage{enumerate}

\theoremstyle{plain}% default
\newtheorem{theorem}{Theorem}
\newtheorem{lemma}{Lemma}[section]
\newtheorem{corollaryP}[lemma]{Corollary}
\newtheorem{propositionP}[lemma]{Proposition}
\newtheorem{proposition}{Proposition}
\newtheorem{corollary}{Corollary}
\theoremstyle{definition}
\newtheorem*{condition}{Condition}

\theoremstyle{remark}
\newtheorem{remark}{Remark}%[section]

\def\R{\ensuremath{\mathbb R}}
\def\N{\ensuremath{\mathbb N}}

\def\I{\ensuremath{{\bf 1}}}
\def\e{\ensuremath{\text e}}
\def\ev{\ensuremath{E\!V}}
\def\S{\ensuremath{\mathcal S}}
\def\RR{\ensuremath{\mathcal R}}

\def\B{\ensuremath{\mathcal BC}}

\def\l{{\rm Leb}}

\def\O{\ensuremath{O}}

\def\n{\ensuremath{n}}

\def\v{\ensuremath{v}}
\def\X{\mathcal{X}}
\def\ie{{\em i.e.,} }
\def\eg{{\em e.g.} }
\def\E{\mathbb E}
\def\eveps{\ensuremath{{\mathcal E}_{\v}(\epsilon)}}

\def\crit{\text{Crit}}
\def\dist{\ensuremath{\mbox{dist}}}
\def\st{such that }

\numberwithin{equation}{section}

\begin{document}

\title[Hitting Times and Extreme Values] {Hitting Time Statistics and Extreme Value Theory}

\author[A. C. M. Freitas]{Ana Cristina Moreira Freitas}
\address{Ana Cristina Moreira Freitas\\ Centro de Matem\'{a}tica \&
Faculdade de Economia da Universidade do Porto\\ Rua Dr. Roberto Frias \\
4200-464 Porto\\ Portugal} \email{amoreira@fep.up.pt}

\author[J. M. Freitas]{Jorge Milhazes Freitas}
\address{Jorge Milhazes Freitas\\ Centro de Matem\'{a}tica da Universidade do Porto\\ Rua do
Campo Alegre 687\\ 4169-007 Porto\\ Portugal}
\email{jmfreita@fc.up.pt}
\urladdr{http://www.fc.up.pt/pessoas/jmfreita}

\author[M. Todd]{Mike Todd}
\address{Mike Todd\\ Centro de Matem\'{a}tica da Universidade do Porto\\ Rua do
Campo Alegre 687\\ 4169-007 Porto\\ Portugal} \email{mtodd@fc.up.pt}
\urladdr{http://www.fc.up.pt/pessoas/mtodd}

\thanks{JMF is partially supported by POCI/MAT/61237/2004 and MT
is supported by FCT grant SFRH/BPD/26521/2006. All three authors are
supported by FCT through CMUP}

\date{\today}

\keywords{Return Time Statistics, Extreme Value Theory, Non-uniform
hyperbolicity, Interval maps} \subjclass[2000]{37A50, 37C40, 60G10,
60G70, 37B20, 37D25, 37E05}

\begin{abstract}
We consider discrete time dynamical systems and show the link
between Hitting Time Statistics (the distribution of the first time
points land in asymptotically small sets) and Extreme Value Theory
(distribution properties of the partial maximum of stochastic
processes). This relation allows to study Hitting Time Statistics
with tools from Extreme Value Theory, and vice versa.  We apply
these results to non-uniformly hyperbolic systems and prove that a
multimodal map with an absolutely continuous invariant measure must
satisfy the classical extreme value laws (with no extra condition on
the speed of mixing, for example).  We also give applications of our
theory to higher dimensional examples, for which we also obtain
classical extreme value laws and exponential hitting time statistics
(for balls).  We extend these ideas to the subsequent returns to
asymptotically small sets, linking the Poisson statistics of both
processes.
\end{abstract}

\maketitle

\section{Introduction}

In this paper we demonstrate and exploit the link between Extreme Value Laws (EVL) and the laws for the Hitting Time Statistics (HTS) for discrete time non-uniformly hyperbolic dynamical systems.

The setting is a discrete time dynamical system $(\X,\mathcal
B,\mu,f)$, where $\X$ is a $d$-dimensional Riemannian manifold,
$\mathcal B$ is the
Borel $\sigma$-algebra, $f:\X\to\X$ is a measurable map
 and $\mu$ an $f$-invariant
probability measure (for all $A\in\mathcal B$ we have
$\mu(f^{-1}(A))=\mu(A)$). We consider a Riemannian metric on $\X$
that we denote by `$\dist$' and for any $\zeta\in\X$ and $\delta>0$,
we define $B_{\delta}(\zeta)=\{x\in\X: \dist(x,\zeta)<\delta\}$.
Also let $\l$ denote Lebesgue measure on $\X$ and for every
$A\in\mathcal B$ we will write $|A|:=\l(A)$. The measure $\mu$ will
be an absolutely continuous invariant probability measure (acip)
with density denoted by $\rho=\frac{d\mu}{d\l}$ . We will denote
$\R^+:=(0,\infty)$ and $\R_0^+:=[0,\infty)$.

\subsection{Extreme Value Laws}

\label{subsec:EVL intro}

In this context, by EVL we mean the study of the asymptotic
distribution of the partial maximum of observable random variables
evaluated along the orbits of the system. To be more precise, take
an observable $\varphi:\X\to\R\cup\{\pm\infty\}$
achieving a global maximum at $\zeta\in \X$ (we allow
$\varphi(\zeta)=+\infty$) and consider the stationary stochastic
process $X_0, X_1,\ldots$ given by
\begin{equation}
\label{eq:def-stat-stoch-proc} X_n=\varphi\circ f^n,\quad \mbox{for
each } n\in {\mathbb N}.
\end{equation}
Define the partial maximum
\begin{equation}
\label{eq:def-max} M_n:=\max\{X_0,\ldots,X_{n-1}\}.
\end{equation}
If $\mu$ is ergodic then Birkhoff's law of large numbers says that
$M_n\to\varphi(\zeta)$ almost surely. Similarly to central limit
laws for partial sums, we are interested in knowing if there are
normalising sequences  $\{a_n\}_{n\in\N}\subset \R^+$ and
$\{b_n\}_{n\in\N}\subset\R$ such that
\begin{equation}
\label{eq:def-EVL} \mu\left(\{x:a_n(M_n-b_n)\le
y\}\right)=\mu\left(\{x:M_n\le u_n\}\right)\to H(y),
\end{equation}
for some non-degenerate distribution function (d.f.) $H$, as
$n\to\infty$.  Here $u_n:=u_n(y)=y/a_n+b_n$ is such that
\begin{equation}
\label{eq:un}
  n\mu(X_0>u_n)\to \tau,\;\mbox{ as $n\to\infty$,}
\end{equation} for some $\tau=\tau(y)\geq0$ and in fact $H(y)=H(\tau(y))$. When this happens we
say that we have an \emph{Extreme Value Law} (EVL) for $M_n$. Note
that, clearly, we must have $u_n\to\varphi(\zeta)$, as $n\to\infty$.
We refer to an event $\{X_j>u_n\}$ as an \emph{exceedance}, at time
$j$, of level $u_n$. Classical Extreme Value Theory asserts that
there are only three types of non-degenerate asymptotic
distributions for the maximum of an independent and identically
distributed (i.i.d.) sample under linear normalisation. They will be
referred to as \emph{classical} EVLs and we denote them by:
\begin{enumerate}[Type 1:]

\item $\ev_1(y)=\e^{-\e^{-y}}$ for $y\in\R$; this is also known
as the \emph{Gumbel} extreme value distribution (e.v.d.).

\item $\ev_2(y)=\e^{-y^{-\alpha}}$, for $y>0$, $\ev_2(y)=0$,
otherwise, where $\alpha>0$ is a parameter; this family of d.f.s is known as the \emph{Fr\'echet} e.v.d.

\item $\ev_3(y)=\e^{-(-y)^{\alpha}}$, for $y\leq0$, $\ev_3(y)=1$,
otherwise, where $\alpha>0$ is a parameter; this family of d.f.s is
known as the \emph{Weibull} e.v.d.
\end{enumerate}
The same limit laws apply to stationary stochastic processes,
under certain conditions on the dependence structure, which allow
the reduction to the independent case. With this in mind, to a given
stochastic process $X_0,X_1,\ldots$ we associate an i.i.d. sequence
$Y_0,Y_1,\ldots$, whose d.f. is the same as that of $X_0$, and whose partial
maximum we define as
\begin{equation}\label{eq:def-max-iid}\hat
M_n:=\max\{Y_0,\ldots,Y_{n-1}\}.
\end{equation}
In the i.i.d. setting, the convergence in \eqref{eq:un} is equivalent to
\begin{equation}
  \label{eq:iid}
  \mu(\hat M_n\leq u_n)\to \e^{-\tau},\;\mbox{ as $n\to\infty$.}
\end{equation}
This is the content of \cite[Theorem~1.5.1]{LLR83} and depending on
the type of limit law that applies, we have that $\tau=\tau(y)$ is
of one of the following three types: $\tau_1(y)=\e^{-y}$ for $y \in
{\mathbb R}$, $\tau_2(y)=y^{-\alpha}$ for $y>0$, and
$\tau_3(y)=(-y)^{\alpha}$ for $y\leq0$.

In the dependent context, the general strategy is to prove that if $X_0,X_1,\ldots$ satisfies
some conditions, then the same limit law for $\hat M_n$ applies
to $M_n$ with the same normalising sequences $\{a_n\}_{n\in\N}$ and
$\{b_n\}_{n\in\N}$. Following \cite{LLR83} we refer to these
conditions as $D(u_n)$ and $D'(u_n)$, where $u_n$ is the sequence of
thresholds appearing in \eqref{eq:def-EVL}. Both conditions impose
some sort of independence but while $D(u_n)$ acts on
 the long range, $D'(u_n)$ is a short
range requirement.

The original condition $D(u_n)$ from \cite{LLR83}, which we will
denote by $D_1(u_n)$, is a type of uniform mixing requirement
specially adapted to Extreme Value Theory. Let
$F_{i_1,\ldots,i_n}(x_1,\ldots,x_n)$ denote the joint d.f. of
$X_{i_1},\ldots,X_{i_n}$, and set
$F_{i_1,\ldots,i_n}(u)=F_{i_1,\ldots,i_n}(u,\ldots,u)$.
\begin{condition}[$D_1(u_n)$]\label{cond:D1}We say that $D_1(u_n)$ holds
for the sequence $X_0,X_1,\ldots$ if for any integers
$i_1<\ldots<i_p$ and $j_1<\ldots<j_k$ for which $j_1-i_p>m$, and any
large $n\in\N$,
\[
\left|F_{i_1,\ldots,i_p,j_1,\ldots,j_k}(u_n)-F_{i_1,\ldots,i_p}(u_n)
F_{j_1,\ldots,j_k}(u_n)\right|\leq \gamma(n,m),
\]
where $\gamma(n,m_n)\xrightarrow[n\to\infty]{}0$, for some sequence
$m_n=o(n)$.
\end{condition}
Since usually the information concerning mixing rates of the systems
is known through decay of correlations, in \cite{FF2} we proposed a
weaker version, which we will denote by $D_2(u_n)$, which still
allows us to relate the distributions of $\hat M_n$ and $M_n$. The
advantage is that it follows immediately from sufficiently fast
decay of correlations for observables which are of bounded variation or
H\"older continuous (see \cite[Section~2]{FF2} and Lemma~\ref{lem:decay-correl-Dun}).

\begin{condition}[$D_2(u_n)$]\label{cond:D2}We say that $D_2(u_n)$ holds
for the sequence $X_0,X_1,\ldots$ if for any integers $\ell,t$
and $n$
\[ \left|\mu\left(\{X_0>u_n\}\cap
  \{\max\{X_{t},\ldots,X_{t+\ell-1}\}\leq u_n\}\right)-\mu(\{X_0>u_n\})
  \mu(\{M_{\ell}\leq u_n\})\right|\leq \gamma(n,t),
\]
where $\gamma(n,t)$ is nonincreasing in $t$ for each $n$ and
$n\gamma(n,t_n)\to0$ as $n\rightarrow\infty$ for some sequence
$t_n=o(n)$.
\end{condition}

By \eqref{eq:un}, the sequence $u_n$ is such that the average number of exceedances in
the time interval $\{0,\ldots,\lfloor n/k\rfloor\}$ is approximately $\tau/k$,
which goes to zero as $k\rightarrow\infty$. However, the exceedances
may have a tendency to be concentrated in the time period following
the first exceedance at time $0$. To avoid this we introduce:

\begin{condition}[$D'(u_n)$]\label{cond:D'un} We say that $D'(u_n)$
holds for the sequence $X_0,X_1,\ldots$ if
\begin{equation}
\label{eq:D'un} \lim_{k\rightarrow\infty}
\limsup_{n\rightarrow\infty}\,n\sum_{j=1}^{\lfloor n/k\rfloor}
\mu(\{X_0>u_n\}\cap
\{X_j>u_n\})=0.
\end{equation}
\end{condition}
This guarantees that the exceedances should appear scattered through
the time period $\{0,\ldots,n-1\}$.

The main result in \cite[Theorem 1]{FF2} states that if $D_2(u_n)$ and
$D'(u_n)$ hold for the process $X_0, X_1,\ldots$ and for a sequence of levels satisfying \eqref{eq:un},
then the following limits exist, and
\begin{equation}
  \label{eq:rel-hatMn-Mn}
  \lim_{n\to\infty}\mu(\hat M_n\leq u_n)=
  \lim_{n\to\infty}\mu(M_n\leq u_n).
\end{equation}
The above statement remains true if we replace $D_2(u_n)$ by
$D_1(u_n)$ (see \cite[Theorem~3.5.2]{LLR83}).

We assume that the observable $\varphi:\X\to\R\cup\{+\infty\}$ is of
the form
\begin{equation}
\label{eq:observable-form} \varphi(x)=g(\dist(x,\zeta)),
\end{equation} where $\zeta$ is a chosen point in the
phase space $\X$ and the function $g:[0,+\infty)\rightarrow {\mathbb
R\cup\{+\infty\}}$ is such that $0$ is a global maximum ($g(0)$ may
be $+\infty$); $g$ is a strictly decreasing bijection $g:V \to W$
in a neighbourhood $V$ of
$0$; and has one of the
following three types of behaviour:
\begin{enumerate}[Type 1:]
\item there exists some strictly positive function
$p:W\to\R$ such that for all $y\in\R$
\begin{equation}\label{eq:def-g1}\displaystyle \lim_{s\to
g_1(0)}\frac{g_1^{-1}(s+yp(s))}{g_1^{-1}(s)}=\e^{-y};
\end{equation}
\item $g_2(0)=+\infty$ and there exists $\beta>0$ such that
for all $y>0$
\begin{equation}\label{eq:def-g2}\displaystyle \lim_{s\to+\infty}
\frac{g_2^{-1}(sy)}{g_2^{-1}(s)}=y^{-\beta};\end{equation}
\item $g_3(0)=D<+\infty$ and there exists $\gamma>0$ such
that for all $y>0$
\begin{equation}\label{eq:def-g3}\lim_{s\to0}
\frac{g_3^{-1}(D-sy)}{g_3^{-1}(D-s)}= y^\gamma.
\end{equation}
\end{enumerate}

Examples of each one of the three types are as follows:
$g_1(x)=-\log x$ (in this case \eqref{eq:def-g1} is easily verified
with $p\equiv1$), $g_2(x)=x^{-1/\alpha}$ for some $\alpha>0$ (condition
\eqref{eq:def-g2} is verified with $\beta=\alpha$) and
$g_3(x)=D-x^{1/\alpha}$ for some $D\in\R$ and $\alpha> 0$ (condition
\eqref{eq:def-g3} is verified with $\gamma=\alpha$).

\begin{remark}
\label{rem:attraction-domains}
  Let the d.f. $F$ be given by
$F(u)=\mu(X_0\leq u)$ and set $u_F=\sup\{y: F(y)<1\}$. Observe that
if at time $j\in\N$ we have an exceedance of the level $u$
(sufficiently large), i.e., $X_j(x)>u$, then we have an entrance of
the orbit of $x$ into the ball $B_{g^{-1}(u)}(\zeta)$ of radius
$g^{-1}(u)$ around $\zeta$, at time $j$. This means that the
behaviour of the tail of $F$, \ie the behaviour of $1-F(u)$ as $u\to
u_F$ is determined by $g^{-1}$, if we assume that Lebesgue's
Differentiation Theorem holds for $\zeta$, since in that case
$1-F(u)\sim \rho(\zeta) |B_{g^{-1}(u)}(\zeta)|$, where
$\rho(\zeta)=\frac{d\mu}{d\l}(\zeta)$. From classical Extreme Value
Theory we know that the behaviour of the tail determines the limit
law for partial maximums of i.i.d. sequences and vice-versa. The
above conditions are just the translation in terms of the shape of
$g^{-1}$, of the sufficient and necessary conditions on the tail of
$F$ of \cite[Theorem~1.6.2]{LLR83}, in order to exist a
non-degenerate limit distribution for $\hat M_n$. In fact, if some
$\ev_i$ applies to $\hat M_n$, for some $i\in\{1,2,3\}$, then $g$
must be of type $g_i$.
\end{remark}

As can be seen from the definitions of $D_2(u_n)$ and $D'(u_n)$,
proving EVLs for absolutely continuous invariant measures for
uniformly expanding dynamical systems is straightforward.
 The study of EVLs for non-uniformly hyperbolic dynamical systems
 has been addressed in the papers \cite{Collexval} and \cite{FF}.

In \cite{Collexval}, Collet considered non-uniformly hyperbolic $C^2$ maps of the
interval which admit an acip $\mu$, with exponential decay of
correlations and obtained a Gumbel EVL for observables of type $g_1$
(actually he took $g_1(x)=-\log x$), achieving a global maximum at
$\mu$-a.e. $\zeta$ in the phase space. We remark that neither the critical points nor its orbits were included in this full $\mu$-measure set of points $\zeta$.

In \cite{FF} the quadratic maps $f_a(x)=1-ax^2$ on $I=[-1,1]$ were considered, with $a\in\B$, where $\B$ is the
Benedicks-Carleson parameter set introduced in \cite{BC85}. For each
map $f_a$ with $a\in\B$, a Weibull EVL was obtained for
observables of type $g_3$ achieving a maximum either at the critical
point or at the critical value.

\subsection{Hitting Time Statistics}
\label{subsec:HTS intro}

We next turn to Hitting Time Statistics for the dynamical system
$(\X,\mathcal B, f,\mu)$.  For a set
$A\subset \X$ we let $r_{A}(y)$ denote the \emph{first hitting time
to $A$} of the point $y$.  That is, the first time $j\ge 1$ so that
$f^j(y) \in A$.  We will be interested in the fluctuations of this functions as the set $A$ shrinks.  Firstly we consider the Return Time Statistics (RTS) of this system.  Let $\mu_A$ denote the conditional measure on $A$, \ie $\mu_A:=\frac{\mu|_A}{\mu(A)}$.  By Kac's Lemma, the expected value of $r_A$ with respect to $\mu$ is  $\int_A r_A~d\mu_A =1/\mu(A)$.  So in studying the fluctuations of $r_A$ on $A$, the relevant normalising factor is $1/\mu(A)$.
Given a sequence of sets $\{U_n\}_{n\in \N}$ so that
$\mu(U_n) \to 0$, the system has \emph{Return Time Statistics}
$G(t)$ for $\{U_n\}_{n\in \N}$ if for all $t\ge 0$ the following limit exists and equals $G(t)$:
\begin{equation}\label{eq:def-RTS-law}\lim_{n\to \infty}
\mu_{U_n}\left(r_{U_n}\geq\frac t{\mu(U_n)}\right).
\end{equation}
We say that $(\X, f,\mu)$ has \emph{Return Time Statistics $G(t)$ to balls at $\zeta$} if for any sequence $\{\delta_n\}_{n\in \N}\subset \R^+$ such that $\delta_n\to 0$ as $n\to \infty$ we have RTS $G(t)$ for $U_n=B_{\delta_n}(\zeta)$.

If we study $r_A$ defined on the whole of $\X$, \ie not simply restricted to $A$, we are studying the Hitting Time Statistics.  Note that we will use the same normalising factor $1/\mu(A)$ in this case. Analogously to the above, given a sequence of sets $\{U_n\}_{n\in \N}$ so that $\mu(U_n) \to 0$, the system has \emph{Hitting Time Statistics}
$G(t)$ for $\{U_n\}_{n\in \N}$ if for all $t\ge 0$ the following limit is
defined and equals $G(t)$:
\begin{equation}\label{eq:def-HTS-law}\lim_{n\to \infty}
\mu\left(r_{U_n}\geq\frac t{\mu(U_n)}\right).
\end{equation}

HTS to balls at a point $\zeta$ is defined analogously to RTS to balls.  In \cite{HLV}, it was shown that the limit for the HTS defined in \eqref{eq:def-HTS-law} exists if and only if the limit for the analogous RTS defined in \eqref{eq:def-RTS-law} exists.  Moreover, they show that  the HTS distribution exists and is exponential (\ie $G(t) = \e^{-t}$) if and only if the RTS distribution exists and is exponential.

For many mixing systems it is known that the HTS are exponential.
For example, this was shown for Axiom A diffeomorphisms in
\cite{Hir}, transitive Markov chains in \cite{Pit} and uniformly
expanding maps of the interval  in \cite{Collint}.  Note that in
these papers the authors were also interested in the (Poisson)
statistics of subsequent returns to some shrinking sets.   For
various results on some systems with some strong hyperbolicity
properties see also \eg \cite{Abadi, Coelho, AbGal}.

For non-uniformly hyperbolic systems less is known.  A major breakthrough in the study of HTS/RTS for non-uniformly hyperbolic maps was made in \cite{HSV}, where they gave a set of conditions which, when satisfied, imply exponential RTS to cylinders and/or balls.  Their principal application was to maps of the interval with an indifferent fixed point.  They also provided similar conditions to imply (Poisson) laws for the subsequent visits of points to shrinking sets. (See Section~\ref{sec:Poisson RTS}).

Another important paper in this direction was \cite{BSTV}, in which they showed that the RTS for a map are the same as the RTS for the first return map.  (The first return map to a set $U\subset \X$ is the map $F=f^{r_U}$.) Since it is often the case that the first return maps for non-uniformly hyperbolic dynamical systems are much better behaved (possibly hyperbolic) than the original system, this provided an extremely useful tool in this theory.  For example, they proved that if $f:I \to I$ is a unimodal map for which the critical point is nowhere dense, and for which an acip $\mu$ exists, then the relevant first return systems $(U,F,\mu_U)$ have a `Rychlik' property.  They then showed that such systems, studied in \cite{Rych}, must have exponential RTS, and hence the original system $(I,f,\mu)$ also has exponential RTS (to balls around $\mu$-a.e. point).

The presence of a recurrent critical point means that the first
return map  itself will not satisfy this Rychlik property. To
overcome this problem in \cite{BrV} special induced maps, $(U,F)$,
were used, where for $x\in U$ we have $F(x)=f^{\text{ind}(x)}(x)$
for some inducing time $\text{ind}(x)\in \N$ that is not necessarily
the first return time of $x$ to $U$.  The fact that these particular
maps can be seen as first return maps in the canonical Markov
extension, the `Hofbauer tower', meant that they were still able to
exploit the main result of \cite{BSTV} to get exponential RTS around
$\mu$-a.e. point for unimodal maps $f:I\to I$ with an acip $\mu$ as
long as $f$ satisfies a polynomial growth condition along the
critical orbit. In \cite{BTintstat} this result was improved to
include any multimodal map with an acip, irrespective of the growth
along the critical orbits, and of the speed of mixing.

We would like to remark that in the case of partially hyperbolic dynamical systems, \cite{Dol} proved exponential RTS, using techniques similar to \cite{Pit}.  In fact the theory there also covers the (Poisson) statistics of subsequent returns to shrinking sets of balls.  These statistics were also considered for toral automorphisms, using a different method, in \cite{DenGoSha}.

We note that for dynamical systems $(\X, \mathcal B, f, \mu)$ where $\mu$ is an equilibrium state, the RTS/HTS to the dynamically defined cylinders are often well understood, see for example \cite{AbGal}.  However, for non-uniformly hyperbolic dynamical systems it is not always possible to go from these strong results to the corresponding results for balls.  We would like to emphasise that in this paper we focus on the HTS to balls, rather than cylinders.

\subsection{Main Results}
\label{subsec:main results}

Our first main result, which obtains EVLs from HTS, is the following.

\begin{theorem}
\label{thm:HTS-implies-EVL} Let $(\X,\mathcal B, \mu,f)$ be a
dynamical system where $\mu$ is an acip, and consider $\zeta\in\X$
for which Lebesgue's Differentiation Theorem holds.
\begin{itemize}

\item If we have HTS to balls centred on $\zeta\in\X$,
then we have an EVL for $M_n$ which applies to the observables
\eqref{eq:observable-form} achieving a maximum at $\zeta$.

\item If we have exponential HTS ($G(t)=\e^{-t}$)
to balls at $\zeta\in\X$, then we have an EVL for $M_n$ which
coincides with that of $\hat M_n$ (meaning that
\eqref{eq:rel-hatMn-Mn} holds). In particular, this EVL must be one
of the 3 classical types. Moreover, if $g$ is of type $g_i$,  for
some $i\in\{1,2,3\}$, then we have an EVL for $M_n$ of type $\ev_i$.

\end{itemize}

\end{theorem}

We next define a class of multimodal interval maps $f:I \to I$. We
denote the finite set of critical points by $\crit$. We say that
$c\in \crit$ is \emph{non-flat} if there exists a diffeomorphism
$\psi_c:\R \to \R$ with $\psi_c(0)=0$ and $1<\ell_c<\infty$ \st for
$x$ close to $c$, $f(x)=f(c)\pm|\psi_c(x-c)|^{\ell_c}$.  The value
of $\ell_c$ is known as the \emph{critical order} of $c$. Let
\[
N\!F^k := \left\{ f:I \to I : f \mbox{ is } C^k, \mbox{ each $c \in
\crit$ is non-flat and } \inf_{f^n(p) = p} |Df^n(p)| > 1 \right\}.
\]

The following is a simple corollary of
Theorem~\ref{thm:HTS-implies-EVL} and \cite[Theorem
3]{BTintstat}.  It generalises the result of Collet in \cite{Collexval}
from unimodal maps with exponential growth on the critical point to
multimodal maps where we only need to know that there is an acip.

\begin{corollary}
Suppose that $f\in N\!F^2$ and $f$ has an acip $\mu$.  Then
$(I,f,\mu)$ has an EVL for $M_n$ which coincides with that of
$\hat M_n$, and this holds for $\mu$-a.e. $\zeta\in\X$ fixed at the
choice of the observable in \eqref{eq:observable-form}. Moreover,
the EVL is of type $\ev_i$ when the observables are of type $g_i$,
for each $i\in\{1,2,3\}$. \label{cor:acip EVL}
\end{corollary}

Now, we state a result in the other direction, \ie we show how to  get HTS from EVLs.

\begin{theorem}
\label{thm:EVL=>HTS} Let $(\X,\mathcal B, \mu,f)$ be a dynamical
system where $\mu$ is an acip and consider $\zeta\in\X$ for which
Lebesgue's Differentiation Theorem holds.
\begin{itemize}

\item If we have an EVL for $M_n$ which applies to the observables
\eqref{eq:observable-form} achieving a maximum at $\zeta\in\X$ then
we have HTS to balls at $\zeta$.

\item If we have an EVL for $M_n$ which coincides with that of $\hat M_n$, then
we have exponential HTS ($G(t)=\e^{-t}$) to balls at $\zeta$.
\end{itemize}
\end{theorem}

The following is immediate by the above and \cite[Theorem 1]{FF2}
(see \eqref{eq:rel-hatMn-Mn}).
\begin{corollary}
  Let $(\X,\mathcal B, \mu,f)$ be a dynamical system where $\mu$ is an
acip and consider $\zeta\in\X$ for which Lebesgue's Differentiation
Theorem holds. If $D_2(u_n)$ (or $D_1(u_n)$) and $D'(u_n)$ hold for
a stochastic process $X_0,X_1,\ldots$ defined by
\eqref{eq:def-stat-stoch-proc} and \eqref{eq:observable-form}, where
$u_n$ is a sequence of levels satisfying \eqref{eq:un}, then we have
exponential HTS to balls at $\zeta$. \label{cor:D+D'=>HTS}
\end{corollary}

The following is an immediate corollary of Theorem~\ref{thm:EVL=>HTS} and the
main theorem of \cite{FF}.

\begin{corollary}
For every Benedicks-Carleson quadratic map $f_a$ (with $a\in\B$) we
have exponential HTS to balls around the critical point or the
critical value. \label{cor:BC}
\end{corollary}

The next result is a byproduct of Theorems~\ref{thm:HTS-implies-EVL},
\ref{thm:EVL=>HTS} and the fact that under $D_1(u_n)$ the only
possible limit laws for partial maximums are the classical
$\ev_i$ for $i\in\{1,2,3\}$.  Since this is not as immediate as the other corollaries, we include a short proof in Section~\ref{sec:proofs HTS EVL}.

\begin{corollary}
Let $(\X,\mathcal B, \mu,f)$ be a dynamical system $\mu$ is an acip
and consider $\zeta\in\X$ for which Lebesgue's Differentiation
Theorem holds. If $D_1(u_n)$ holds for a stochastic process
$X_0,X_1,\ldots$ defined by \eqref{eq:def-stat-stoch-proc} and
\eqref{eq:observable-form}, where $u_n$ is a sequence of levels
satisfying \eqref{eq:un}, then the only possible HTS to balls around
$\zeta$ are of exponential type, meaning that, there is $\theta>0$
such that $G(t)=\e^{-\theta t}$. \label{cor:only HTS in town}
\end{corollary}

Note that for this corollary to be non-trivial, we must assume that there exists a distribution for HTS.  This may not always be the case. For example, in \cite{CoedeF, Coe} it was shown that for certain circle diffeomorphisms there are sequences of intervals $\{U_{n}\}_{n\in \N}, \ \{V_{n}\}_{n\in \N}$ which both shrink to the same point $\zeta$, but yield different HTS laws.  Note that in these cases $D_1(u_n)$ also fails.

As we have already mentioned, Corollary~\ref{cor:acip EVL}
generalises the result of Collet in \cite{Collexval}, which was for
$C^2$ non-uniformly hyperbolic maps  of the interval (admitting a
Young tower). However, a close look to Collet's arguments allows us
to conclude that his result still prevails in higher dimensions. In
fact, one can show that if we consider non-uniformly expanding maps
(in any finite dimensional compact manifold), admitting a so-called
Young tower with exponential return times to the base, then for any
sequence of r.v. $X_0, X_1,\ldots$, defined as in
\eqref{eq:def-stat-stoch-proc} and for a sequence of levels $u_n$
such that $n\mu(X_0>u_n)\to\tau>0$, conditions $D_2(u_n)$ and
$D'(u_n)$ hold. This means that by the above theorems, we can prove
both EVLs and HTS for these maps. Due to numerous definitions
required for that setting, we leave both the theorems and the proofs
on this subject to Section~\ref{sec:EVL-HTS-higher-dim}.

Theorems~\ref{thm:HTS-implies-EVL} and \ref{thm:EVL=>HTS} give us
new tools to investigate the recurrence of dynamical systems,
principally by allowing us to use the wealth of theory for HTS which
has been developed in recent years to prove EVLs.  We note that in
Corollary~\ref{cor:acip EVL}, the dynamical systems involved need
not have any fast rate of decay of correlations at all.  Indeed, a
priori the relevant system may only have summable decay of
correlations. As in Section~\ref{sec:EVL-HTS-higher-dim} where we
consider higher dimensional maps admitting Young towers, there are
situations where it is actually easier to check conditions like
$D_2(u_n)$ and $D'(u_n)$ in order to get laws for HTS. In fact, to
our knowledge, exponential HTS to balls have never been proved
before for higher dimensional non-uniformly expanding systems: in
such cases, inducing schemes with the nice properties of
one-dimensional dynamics are much harder to find. Also the dynamical
systems we present in this paper should provide models which can be
used in investigating Extreme Value Theory both analytically and
numerically. Namely, the simple fact that we get EVLs from
deterministic models may be an extra advantage for numerical
simulation since there is no need to generate random numbers. This
means that this theory may reveal very useful for testing GEV
(Generalised Extreme Value distribution) fitting for data
corresponding to phenomena for which there is an underlying
deterministic model.

The next question that arises is: what about subsequent visits to
$U_n$ or subsequent exceedances of the level $u_n$? Namely, we are
interested in the \emph{point processes} associated to the instants
of occurrence of returns to $U_n$ and exceedances of the level
$u_n$. If we have either exponential HTS or a classical EVL then
time between hits or exceedances is exponentially distributed. This
means that we should expect a Poisson limit for the point processes.
We show in Section~\ref{sec:link-HTPP-EPP} that the relation between
HTS and EVL does indeed extend to the laws for the subsequent
visits/exceedances (we postpone the precise definitions and results
to Sections~\ref{sec:link-HTPP-EPP}, \ref{sec:PoissonEVL} and
\ref{sec:Poisson RTS}).  More precisely, we show that the point
process of hitting times has a Poisson limit if and only if the
point process of exceedances has a Poisson limit. We next discuss
how to obtain a Poisson law in these two different contexts.  In
Section~\ref{sec:PoissonEVL} we give conditions which guarantee a
Poisson limit for the point process of exceedance times. This part
of the paper can be seen as  a generalisation of \cite{FF2}.
Moreover, we show that these conditions can be verified in the
settings from \cite{Collexval,FF}, leading to Poisson statistics for
both point processes for the systems considered.  In
Section~\ref{sec:Poisson RTS} we show that in many cases for
multimodal maps it can be shown that the HTS behave asymptotically
as a Poisson distribution.

Throughout this paper the notation $A_n\sim B_n$ means that
$\lim_{n\to \infty} \frac{A_n}{B_n}=1$. Also, if
$\{\delta_n\}_{n\in\N}\subset \R^+$ has $\delta_n\to 0$ as
$n\to\infty$, then for each $\zeta\in\X$, let $\kappa\in (0,\infty)$
be such that $|B_{\delta_n}(\zeta)|\sim \kappa \cdot\delta_n^d$. Let
$x\in {\mathbb R}$. We denote the integer part of $x$ by $\lfloor
x\rfloor$ and define $\lceil x\rceil:=x$ if $x=\lfloor x\rfloor$,
and $\lceil x\rceil:=\lfloor x\rfloor+1$ otherwise.

\subsection*{Acknowledgements}
We would like to thank J.F.\ Alves for useful suggestions regarding
the example of a non-uniformly expanding system given in
Section~\ref{subsec:example}.

\section{Proofs of our results on HTS and EVL}
\label{sec:proofs HTS EVL}

In this section we prove Theorems~\ref{thm:HTS-implies-EVL},
\ref{thm:EVL=>HTS} and Corollary~\ref{cor:only HTS in town}.

\begin{proof}[Proof of Theorem~\ref{thm:HTS-implies-EVL}]
Let $\rho(\zeta)=\frac{d\mu}{d\l}(\zeta)\in\R_0^+$ and set
\begin{align*}
u_n&=g_1\left((\kappa\rho(\zeta)n)^{-1/d}\right)
+p\left(g_1\left(\left({\kappa\rho(\zeta)n}\right)^{-1/d}\right)
\right)\frac{y}{d},&&\mbox{for $y\in {\mathbb R}$, for type
$g_1$;}\\
 u_n&=g_2\left(\left({\kappa\rho(\zeta)n}\right)^{-1/d}\right)y,&&\mbox{for $y>0$, for type
$g_2$;}\\
u_n&=D-\left(D-g_3\left(\left({\kappa\rho(\zeta)n}\right)^{-1/d}\right)\right)(-y),
&&\mbox{for $y<0$, for type $g_3$.}
\end{align*}

Note that, for $n$ sufficiently large
\begin{eqnarray}
&&\{x:M_n(x)\leq u_n\}=\bigcap_{j=0}^{n-1}\{x:X_j(x)\leq u_n\}
=\bigcap_{j=0}^{n-1}\{x:g(\dist(f^j(x),\zeta))\leq u_n\}\nonumber\\
&&=\bigcap_{j=0}^{n-1}\{x:\dist(f^j(x),\zeta)\geq g^{-1}(u_n)\}=
\{x:r_{B_{g^{-1}(u_n)}(\zeta)}(x)\geq n\} \label{eq:rel-max-returns}
\end{eqnarray}

Now, observe that \eqref{eq:def-g1}, \eqref{eq:def-g2} and \eqref{eq:def-g3} imply
\begin{align*}
g_1^{-1}(u_n)&=g_1^{-1}\left[g_1\left(\left({\kappa\rho(\zeta)n}
\right)^{-1/d}\right)
+p\left(g_1\left(\left({\kappa\rho(\zeta)n}\right)^{-1/d}\right)
\right)\frac{y}{d}\right]\\&\sim g_1^{-1}
\left[g_1\left(\left({\kappa\rho(\zeta)n}
\right)^{-1/d}\right)\right]\e^{-y/d}
=\left(\frac{\e^{-y}}{\kappa\rho(\zeta)n}\right)^{1/d};\\
g_2^{-1}(u_n)&=g_2^{-1}
\left[g_2\left(\left({\kappa\rho(\zeta)n}\right)^{-1/d}
\right)y\right]\sim g_2^{-1}
\left[g_2\left(\left({\kappa\rho(\zeta)n}\right)^{-1/d}\right)
\right]y^{-\beta}
=\left(\frac{y^{-\beta d}}{\kappa\rho(\zeta)n}\right)^{1/d};\\
g_3^{-1}(u_n)&=g_3^{-1}
\left[D-\left(D-g_3\left(\left({\kappa\rho(\zeta)n}\right)^{-1/d}\right)
\right)(-y)\right]
\\&\sim g_3^{-1}\left[
D-\left(D-g_3\left(\left({\kappa\rho(\zeta)n}\right)^{-1/d}
\right)\right)\right] (-y)^{\gamma}=\left(\frac{(-y)^{\gamma
d}}{\kappa\rho(\zeta)n}\right)^{1/d}.
\end{align*}
Thus, we may write
$$g^{-1}(u_n)\sim\left(\frac{\tau(y)}{\kappa\rho(\zeta)n}\right)^{1/d},$$
meaning that
$$g_i^{-1}(u_n)\sim\left(\frac{\tau_i(y)}{\kappa\rho(\zeta)n}\right)^{1/d},
 \;\; \forall i\in \{1,2,3\}$$
where $\tau_1(y)=\e^{-y}$ for $y \in {\mathbb R}$, $\tau_2(y)=
y^{-\beta d}$ for $y>0$, and $\tau_3(y)=(-y)^{\gamma d}$ for $y<0$.

Since Lebesgue's Differentiation Theorem holds for
$\zeta\in \X$, we have
$\frac{\mu(B_\delta(\zeta))}{|B_\delta(\zeta)|} \to \rho(\zeta)$ as
$\delta\to 0$. Consequently, since it is obvious that
$g^{-1}(u_n)\to 0$ as $n\to\infty$, then
\[
\mu\left(B_{g^{-1}(u_n)}(\zeta)\right)\sim
\rho(\zeta)|B_{g^{-1}(u_n)}(\zeta)|\sim
\rho(\zeta)\kappa(g^{-1}(u_n))^d=\rho(\zeta)\kappa
\frac{\tau(y)}{\kappa\rho(\zeta)n}=\frac{\tau(y)}{n}.
\]

Thus, we have
\begin{equation}
\label{eq:n-estimate}n\sim
\frac{\tau(y)}{\mu\left(B_{g^{-1}(u_n)}(\zeta)\right)}.
\end{equation}

 Now, we claim that using \eqref{eq:rel-max-returns} and
\eqref{eq:n-estimate}, we have
\begin{align}
\label{eq:approximation-1} \lim_{n\to\infty}\mu(\{x:M_n(x)\leq
u_n\})&=\lim_{n\to\infty}
\mu\left(\left\{x:r_{B_{g^{-1}(u_n)}(\zeta)}(x)\geq
\frac{\tau(y)}{\mu\left(B_{g^{-1}(u_n)}(\zeta)\right)}\right\}\right)\\
&= G(\tau(y))\label{eq:conclusion-1},
\end{align}
which gives the first part of the theorem.

To see that \eqref{eq:approximation-1} holds, observe that by
\eqref{eq:rel-max-returns} and \eqref{eq:n-estimate} we have
\begin{multline*}
 \left|\mu(\{M_n\leq
u_n\})-\mu\left(\left\{r_{B_{g^{-1}(u_n)}(\zeta)}\geq
\frac{\tau(y)}{\mu\left(B_{g^{-1}(u_n)}(\zeta)\right)}\right\}\right)\right|
\\ =
\left| \mu\left(\left\{r_{B_{g^{-1}(u_n)}(\zeta)}\geq
n\right\}\right)-\mu\left(\left\{r_{B_{g^{-1}(u_n)}(\zeta)}\geq
(1+\varepsilon_n)n\right\}\right)\right|,
\end{multline*}
where $\{\varepsilon_n\}_{n\in\N}$ is such that $\varepsilon_n\to0$
as $n\to\infty$. Since we have
\begin{equation}
\label{eq:diff-r-m-k}\left\{r_{B_{g^{-1}(u_n)}(\zeta)}\geq
m\right\}\setminus\left\{r_{B_{g^{-1}(u_n)}(\zeta)}\geq
m+k\right\}\subset \bigcup_{j=m}^{m+k-1} f^{-j}\left(
B_{g^{-1}(u_n)}(\zeta)\right),\;\mbox{$\forall
m,k\in\N$,}\end{equation} it follows by stationarity that
\begin{multline*}
 \left|
\mu\left(\left\{r_{B_{g^{-1}(u_n)}(\zeta)}\geq
n\right\}\right)-\mu\left(\left\{r_{B_{g^{-1}(u_n)}(\zeta)}\geq
(1+\varepsilon_n)n\right\}\right)\right|\\ \leq |\varepsilon_n|n\mu\left(
B_{g^{-1}(u_n)}(\zeta)\right)\sim |\varepsilon_n|\tau \to 0,
\end{multline*}
as $n\to\infty$, completing the proof of \eqref{eq:approximation-1}.

Next we will use the exponential HTS hypothesis, that is
$G(t)=\e^{-t}$, to show the second part of the theorem.

Under the exponential HTS assumption, by \eqref{eq:conclusion-1}
it follows immediately that $\lim_{n\to\infty}\mu(\{x:M_n(x)\leq
u_n\})=\e^{-\tau(y)}$. Now, recalling that in the i.i.d setting
\eqref{eq:un} is equivalent to \eqref{eq:iid} then we also have
$\lim_{n\to\infty}\mu(\{x:\hat M_n(x)\leq u_n\})=\e^{-\tau(y)}$.
As explained in the introduction, this means that $G(\tau)$ must be of the three classical types.

It remains to show that if the observable is of type $g_i$ then the
EVL that applies to $M_n$ is of type $\ev_i$, for each
$i\in\{1,2,3\}$.

\textbf{Type $g_1$:} In this case we have
$\e^{-\tau_1(y)}=\e^{-\e^{-y}}$, for all $y\in {\mathbb R}$, that
corresponds to the Gumbel e.v.d. and so we have an EVL for $M_n$ of
type $\ev_1$.

\textbf{Type $g_2$:} We obtain $\e^{-\tau_2(y)}=\e^{-y^{-\beta d}}$
for $y>0$.  To conclude that in this case we have the Fr\'{e}chet
e.v.d. with parameter $\beta d$, we only have to check that for
$y\leq0$, $\mu(\{x:M_n(x)\leq u_n\})=0$. Since
$g_2\left(\left({\kappa\rho(\zeta)n}\right)^{-1/d} \right)>0$ (for
all large $n$) and $$\mu(\{x:M_n(x)\leq
u_n\})=\mu\left(\left\{x:M_n(x)\leq
g_2\left(\left({\kappa\rho(\zeta)n}\right)^{-1/d}
\right)y\right\}\right)\rightarrow\e^{-y^{-\beta d}}$$ as
$n\to\infty$. Letting $y\downarrow 0$, it follows that
$\mu(\{x:M_n(x)\leq 0\})\rightarrow 0$, and, for $y<0$,
$$\mu(\{x:M_n(x)\leq u_n\})=\mu\left(\left\{x:M_n(x)\leq
g_2\left(\left({\kappa\rho(\zeta)n}\right)^{-1/d} \right)y\right\}\right)\leq
\mu(\{x:M_n(x)\leq 0\})\rightarrow 0.$$ So, we have, in this case, an EVL for $M_n$  of type $\ev_2$.

\textbf{Type $g_3$:} For $y<0$, we have
$\e^{-\tau_3(y)}=\e^{-(-y)^{\gamma d}}$.  To conclude that in this
case we have the Weibull e.v.d. with parameter $\gamma d$, we only
need  to check that for $y\geq0$, $\mu(\{x:M_n(x)\leq u_n\})=1$. In
fact, for $y\geq0$, since
$D-g_3\left(\left({\kappa\rho(\zeta)n}\right)^{-1/d} \right)>0$, we
have \begin{align*} \mu(\{x:M_n(x)\leq u_n\})&
=\mu\left(\left\{x:M_n(x)\leq
\left(D-g_3\left(\left({\kappa\rho(\zeta)n}\right)^{-1/d}\right)
\right)y+D\right\}\right)\\
&\geq \mu(\{x:M_n(x)\leq D\})=1.\end{align*} So we have, in this
case, an EVL for $M_n$  of type $\ev_3$.
\end{proof}

\begin{proof}[Proof of Theorem~\ref{thm:EVL=>HTS}.]
We assume that by hypothesis for every $y\in\R$ and some sequence
$u_n=u_n(y)$ such that $n \mu\left(\{x:\varphi(x)>
u_n(y)\}\right)\xrightarrow[n\to\infty]{}\tau(y)$,  we have
$$\lim_{n\to\infty}\mu\left(\{x:M_n(x)\leq
u_n(y)\}\right)=H(\tau(y)).$$ Given $t>0$ and a sequence
$\{\delta_n\}_{n\in\N}\subset \R^+$ with
$\delta_n\xrightarrow[n\to\infty]{}0$, we take $y\in\R$ such that
$t=\tau(y)$ and define $\ell_n:=\lfloor
t/(\kappa\rho(\zeta)\delta_n^d)\rfloor$. We can always find such $y$
because \eqref{eq:iid} is equivalent to \eqref{eq:un} and $\varphi$
is of the form \eqref{eq:observable-form}, where $g$ is of type
$g_i$, for some $i\in\{1,2,3\}$, which implies that $\hat M_n$ has a
limit law of type $\ev_i$.

First we show that
\begin{equation}
\label{eq:aux-1} g^{-1}\left(u_{\ell_n}\right)\sim\delta_n.
\end{equation}
If $n$ is sufficiently large, then
\[
\{x:\varphi(x)>u_n\}=\{x:g(\dist(x,\zeta))>u_n\}=\{x:\dist(x,\zeta)<g^{-1}
(u_n)\}= B_{g^{-1}(u_n)}(\zeta).
\]
Hence, by assumption on the sequence $u_n$, we have
$n\mu\left(B_{g^{-1}(u_n)}(\zeta)\right)\xrightarrow[n\to\infty]{}
\tau(y)=t$. As Lebesgue's
Differentiation Theorem holds for $\zeta\in \X$, we have
$\frac{\mu(B_\delta(\zeta))}{|B_\delta(\zeta)|} \to \rho(\zeta)$ as
$\delta\to 0$. Consequently, since it is obvious that
$g^{-1}(u_n)\to 0$ as $n\to\infty$, then
$n\left|B_{g^{-1}(u_n)}(\zeta)\right|\xrightarrow[n\to\infty]{}
t/\rho(\zeta)$. Thus, we may write $g^{-1}(u_n)\sim \left(\frac
t{\kappa n\rho(\zeta)}\right)^{1/d}$ and substituting $n$ by $\ell_n$ we are immediately
led to \eqref{eq:aux-1} by definition of $\ell_n$.

Next, using Lebesgue's Differentiation Theorem, again, we get
$\mu\left(B_{\delta_n}(\zeta)\right)\sim\rho(\zeta)\kappa\delta_n^d$
which easily implies that by definition of $\ell_n$,
\begin{equation}
\label{eq:aux-2} \frac t{\mu\left(B_{\delta_n}(\zeta)\right)}\sim \ell_n.
\end{equation}

Now we note that, as in \eqref{eq:rel-max-returns}
\begin{eqnarray}
\label{eq:Mln-rgeqln}
&&\{x:M_{\ell_n}(x)\leq u_{\ell_n}\}=\bigcap_{j=0}^{\ell_n-1}\{x:X_j(x)\leq
u_{\ell_n}\}
=\bigcap_{j=0}^{\ell_n-1}\{x:g(\dist(f^j(x),\zeta))\leq u_{\ell_n}\}\nonumber\\
&&=\bigcap_{j=0}^{\ell_n-1}\{x:\dist(f^j(x),\zeta)\geq
g^{-1}(u_{\ell_n})\}= \{x:r_{B_{g^{-1}(u_{\ell_n})}(\zeta)}(x)\geq
\ell_n\}.
\end{eqnarray}
At this point, we claim that
\begin{equation}
\label{eq:approximation-2}
\lim_{n\to\infty}\mu\left(\left\{x:r_{B_{\delta_n}(\zeta)}(x)\geq\frac t
{\mu(B_{\delta_n}(\zeta))}\right\}\right)=\lim_{n\to\infty}
\mu\left(\{x:M_{\ell_n}(x)\leq u_{\ell_n}\}\right).
\end{equation}
Then, the first part of the theorem follows, once we observe that,
by hypothesis, we have $$\mu\left(\{x:M_{\ell_n}(x)\leq
u_{\ell_n}\}\right)\xrightarrow[n\to\infty]{}H(\tau(y))=H(t).$$

The second part also follows since when the EVL of $M_n$ coincides
with that of $\hat M_n$, then $H(\tau(y))=\e^{-\tau(y)}$. This is
because in the i.i.d. setting \eqref{eq:un} is equivalent to
\eqref{eq:iid}, as we have already mentioned.

It remains to show that \eqref{eq:approximation-2} holds. First, observe that
\begin{align*}
  \mu\left(\left\{r_{B_{\delta_n}(\zeta)}\geq\frac t
{\mu(B_{\delta_n}(\zeta))}\right\}\right)&=
\mu\left(\{M_{\ell_n}\leq u_{\ell_n}\}\right)+
\left(\mu\left(\left\{r_{B_{\delta_n}(\zeta)}\geq \ell_n\right\}\right)-
\mu\left(\{M_{\ell_n}\leq u_{\ell_n}\}\right)\right)\\
&\quad+\left(\mu\left(\left\{r_{B_{\delta_n}(\zeta)}\geq\frac t
{\mu(B_{\delta_n}(\zeta))}\right\}\right)-\mu\left(\left\{r_{B_{\delta_n}(\zeta)}\geq
\ell_n\right\}\right)\right).
\end{align*}

For the third term on the right, note that by \eqref{eq:aux-2} we
have
\begin{multline*}
 \left|\mu\left(\left\{r_{B_{\delta_n}(\zeta)}\geq \ell_n\right\}\right)-
 \mu\left(\left\{r_{B_{\delta_n}(\zeta)}\geq\frac t
{\mu(B_{\delta_n}(\zeta))}\right\}\right) \right|
\\ =
\left| \mu\left(\left\{r_{B_{\delta_n}(\zeta)}\geq
\ell_n\right\}\right)- \mu\left(\left\{r_{B_{\delta_n}(\zeta)}\geq
(1+\varepsilon_n)\ell_n\right\}\right)\right|,
\end{multline*}
for some sequence $\{\varepsilon_n\}_n\in\N$ such that
$\varepsilon_n\to0$, as $n\to\infty$. By \eqref{eq:diff-r-m-k},
\eqref{eq:aux-2} and stationarity it follows that
\begin{equation*}
 \left| \mu\left(\left\{r_{B_{\delta_n}(\zeta)}\geq \ell_n\right\}\right)-
 \mu\left(\left\{r_{B_{\delta_n}(\zeta)}\geq (1+\varepsilon_n)\ell_n\right\}
 \right)\right|
  \leq |\varepsilon_n|\ell_n\mu\left(
B_{\delta_n}(\zeta)\right)\sim|\varepsilon_n|t \to 0,
\end{equation*}
as $n\to\infty$.

For the remaining term, using \eqref{eq:aux-1}, \eqref{eq:aux-2} and
\eqref{eq:Mln-rgeqln}, we have
\begin{align*}
\left|\mu\left(\left\{r_{B_{\delta_n}(\zeta)}\geq \ell_n\right\}\right)-\mu\left(\{M_{\ell_n}\leq u_{\ell_n}\}\right)\right|&=
\left|\mu\left(\left\{r_{B_{\delta_n}(\zeta)}\geq \ell_n\right\}\right)-\mu\left(\{r_{B_{g^{-1}(u_{\ell_n})}(\zeta)}\geq \ell_n\}\right)\right|\\
&\leq \sum_{i=1}^{\ell_n}\mu\left(f^{-i}\left(B_{\delta_n}(\zeta)\bigtriangleup
B_{g^{-1}(u_{\ell_n})}(\zeta)\right)\right)\\
&=\ell_n\mu\left(B_{\delta_n}(\zeta)\bigtriangleup
B_{g^{-1}(u_{\ell_n})}(\zeta)\right)\\
&\sim \frac t{\mu\left(B_{\delta_n}(\zeta)\right)}\left|\mu\left(B_{\delta_n}(\zeta)\right)-
\mu\left(B_{g^{-1}(u_{\ell_n})}(\zeta)\right)\right|\\
&=t\left|1-\frac{\mu\left(B_{g^{-1}(u_{\ell_n})}(\zeta)\right)}{
\mu\left(B_{\delta_n}(\zeta)\right)}\right|\to0
\end{align*}
as $n\to\infty$, which ends the proof of \eqref{eq:approximation-2}.
\end{proof}

\begin{proof}[Proof of Corollary~\ref{cor:only HTS in town}]
Let us assume the existence of HTS to balls around $\zeta$ (not
necessarily exponential). Then the first part of
Theorem~\ref{thm:HTS-implies-EVL} assures the existence of an EVL as
in \eqref{eq:def-EVL} for $M_n$ defined in \eqref{eq:def-max}. This
fact and the hypothesis that $D_1(u_n)$ holds allows us to use
\cite[Theorem~3.7.1]{LLR83} to conclude that there is $\theta>0$
such that $\lim_{n\to\infty}\mu(M_n\leq u_n)=\e^{-\theta \tau}$.
Finally, we use the first part of Theorem~\ref{thm:EVL=>HTS} to
conclude that we have HTS to balls centred on $\zeta$ of exponential
type.
\end{proof}

\section{Relation between hitting times and exceedance point
processes}\label{sec:link-HTPP-EPP}

We have already seen how to relate HTS and EVL.  We next show that if we enrich the process and the statistics by
considering either multiple returns or multiple exceedances we can
take the parallelism even further.

Given a sequence $\{\delta_n\}_{n\in\N}\subset \R^+$ such that
$\delta_n\xrightarrow[n\to\infty]{}0$, for each $j\in\N$, we define
the $j$-th \emph{waiting (or inter-hitting) time} as
\begin{equation}\label{eq:def-multiple-returns}
w^j_{B_{\delta_n}(\zeta)}(x)=
r_{B_{\delta_n}(\zeta)}\left(f^{w^1_{B_{\delta_n}(\zeta)}(x)
+\cdots+w^{j-1}_{B_{\delta_n}(\zeta)}(x) }(x)\right),
\end{equation}
and the \emph{$j$-th hitting time} as
\[
r^j_{B_{\delta_n}(\zeta)}(x)=\sum_{i=1}^{j}w^i_{B_{\delta_n}(\zeta)}(x).
\]
We define the \emph{Hitting Times Point Process} (HTPP) by
counting the number of hitting times during the time interval
$[0,t)$. However, since $\mu(B_{\delta_n}(\zeta))\to 0$, as
$n\to\infty$, then by Kac's Theorem, the expected waiting time
between hits is diverging to $\infty$ as $n$ increases. This fact
suggests a time re-scaling using the factor
$v^*_n:=1/\mu(B_{\delta_n}(\zeta))$, which is precisely the expected
inter-hitting time. Hence, for any $x\in\X$ and every $t\geq0$ define
\begin{equation}
\label{eq:def-HTPP}
N^*_n(t)=N^*_n([0,t),x):=\sup\left\{j:\,r^j_{B_{\delta_n}(\zeta)}(x)\leq
v^*_n t\right\}=\sum_{j=0}^ {\lfloor v^*_n
t\rfloor}\I_{B_{\delta_n}(\zeta)}\circ f^{j}
\end{equation}
When $x\in B_{\delta_n}(\zeta)$ and we consider the conditional
measure $\mu_{B_{\delta_n}(\zeta)}$ instead of $\mu$, then we refer
to $N^*_n(t)$ as the \emph{Return Times Point Process} (RTPP).

If we have exponential HTS, ($G(t)=\e^{-t}$ in
\eqref{eq:def-RTS-law}), then the distribution of the waiting time
before hitting $B_{\delta_n}(\zeta)$ is asymptotically exponential.
Also, if we assume that our systems are mixing, because in that case
we can think that the process gets renewed when we come back to
$B_{\delta_n}(\zeta)$, then one may look at the hitting times as the
sum of almost independent r.v.s that are almost exponentially
distributed. Hence, one would expect that the hitting times, when
properly re-scaled, should form a point process with a Poisson type
behaviour at the limit.

As discussed in Section~\ref{subsec:HTS intro}, for hyperbolic
systems,  it is indeed the case that we do get a Poisson Process as
the limit of HTPP.  The theory in \cite{HSV, BSTV} and
\cite{BTintstat} implies that if $f\in N\!F^2$ has an acip then we
have a Poisson limit for the HTPP.  We postpone a sketch of this
fact to Section~\ref{sec:Poisson RTS}, in order to keep our focus on
the relation between HTS and EVL here.  However, we would like to
remark that a key difference between proofs for the first hitting
time and for showing that we have a Poisson Point Process, if we are
using the theory started in \cite{HSV}, is that a further mixing
condition is required.

Now, we turn to an EVL point of view. In this context, one is
concerned with the occurrence of exceedances of the level $u_n$ for
the stationary stochastic process $X_0,X_1,\ldots$.  In particular,
we are interested in counting the number of exceedances, among a
random sample $X_0,\ldots,X_{n-1}$ of size $n$.  As in the previous
sections, we consider the stationary stochastic process defined by
\eqref{eq:def-stat-stoch-proc} and a sequence of levels
$\{u_n\}_{n\in\N}$ such that $n\mu(X_0>u_n)\to\tau>0$, as
$n\to\infty$. We define the \emph{exceedance point process} (EPP) by
counting the number of exceedances during the time interval $[0,t)$.
We re-scale time using the factor $v_n:=1/\mu(X>u_n)$ given by Kac's
Theorem, again. Then for any $x\in\X$ and every $t\geq0$, set
\begin{equation}
\label{eq:def-EPP} N_n(t)=N_n([0,t),x):=\sum_{j=0}^ {\lfloor v_n
t\rfloor}\I_{X_j>u_n}.
\end{equation}
The limit laws for these point processes can be used to assess the
impact and damage caused by rare events since they describe their
time occurrences, their individual impacts and accumulated effects.
Assuming that the process is mixing, we almost have a situation of
many Bernoulli trials where the expected number of successes is
almost constant ($n\mu(X>u_n)\to\tau>0$). Thus, we expect a Poisson
law as a limit. In fact, one should expect that the exceedance
instants, when properly normalised, should form a point process with
a Poisson Process as a limit, also. This is the content of
\cite[Theorem~5.2.1]{LLR83} which states that under $D_1(u_n)$ and
$D'(u_n)$, the EPP $N_n$, when properly normalised, converges in
distribution to a Poisson Process. (See \cite[Chapter~5]{LLR83},
\cite{HHL} and references therein for more information on the
subject).

Similarly to Theorems~\ref{thm:HTS-implies-EVL} and
\ref{thm:EVL=>HTS}, we show that if there exists a limiting
continuous time stochastic process for the HTPP, when properly
normalised, then the same holds for the EPP and vice-versa. In the
sequel $\xrightarrow[]{d}$ denotes convergence in distribution.

\begin{theorem}
  \label{thm:HTPP=>EPP}Let $(\X,\mathcal B, \mu,f)$ be a
dynamical system where $\mu$ is an acip and consider $\zeta\in\X$
for which Lebesgue's Differentiation Theorem holds.
  Suppose that for any sequence
  $\delta_n\xrightarrow[n\to\infty]{}0$ we have that the HTPP
   defined in \eqref{eq:def-HTPP} is \st
  $N^*_n\xrightarrow[n\rightarrow\infty]{d}N$, where
   $N$ is a
  continuous time stochastic process. Then, for the
  EPP defined in \eqref{eq:def-EPP} we also have
  $N_n\xrightarrow[n\rightarrow\infty]{d}N$.
\end{theorem}

\begin{proof}
  The result follows immediately once we set $\delta_n=g^{-1}(u_n)$
  and observe that for every $j,n\in\N$ and $x\in\X$
  we have $\{x:\,X_j>u_n\}=
  \{x:\,f^j(x)\in B_{g^{-1}(u_n)}(\zeta)\}$, which implies that
  $N_n(t)=N^*_n(t)$, for all $t\geq0$.
\end{proof}

\begin{corollary}
Suppose that $f\in N\!F^2$ and $f$ has an acip $\mu$.
Then, denoting by $N_n$ the associated EPP
as in \eqref{eq:def-EPP}, we have $N_n\xrightarrow[]{d}N$,
as $n\rightarrow\infty$, where $N$
denotes a Poisson Process with intensity $1$.
\label{cor:EPP for acip}
\end{corollary}

The fact that the maps in this corollary satisfy the conditions of Theorem~\ref{thm:HTPP=>EPP} follows from the sketch in Section~\ref{sec:Poisson RTS}.  So the result is otherwise immediate.

\begin{theorem}
  \label{thm:EPP=>HTPP}Let $(\X,\mathcal B, \mu,f)$ be a
dynamical system where $\mu$ is an acip and consider $\zeta\in\X$
for which Lebesgue's Differentiation Theorem holds. Suppose that for
a sequence of levels $\{u_n\}_{n\in\N}$ such that
$n\mu(X_0>u_n)\to\tau>0$, as $n\to\infty$, the EPP defined in
\eqref{eq:def-EPP} is such that
  $N_n\xrightarrow[n\rightarrow\infty]{d}N$, where
    $N$ is a continuous time stochastic process. Then, for the
  HTPP defined in \eqref{eq:def-EPP} we also have
  $N^*_n\xrightarrow[n\rightarrow\infty]{d}N$.
\end{theorem}

\begin{proof}
  Given a sequence $\{\delta_n\}_{n\in \N}\subset \R^+$ with $\delta_n\xrightarrow[n\to\infty]{}0$ we define,
  as in the proof of Theorem~\ref{thm:EVL=>HTS}, the sequence
  $\ell_n$ such that $\delta_n\sim g^{-1}\left(u_{\ell_n}\right)$. Set
  $k_n:=\max\{v_n^*,v_{\ell_n}\}$ and
   observe that $|N_n^*(t)-N_{\ell_n}(t)|\leq \sum_{j=0}^{k_n}
   \I_{B_{\delta_n}(\zeta)\triangle B_{g^{-1}(u_{\ell_n})}(\zeta)}
   \circ f^j.$ Using stationarity we get
   \begin{align*}
    \mu\left(|N_n^*(t)-N_{\ell_n}(t)|>0\right)&\leq k_n
    \mu\left(B_{\delta_n}(\zeta)\triangle B_{g^{-1}(u_{\ell_n})}
    (\zeta)\right)\\&= k_n\left|\mu\left(B_{\delta_n}(\zeta)\right)-
    \mu\left(B_{g^{-1}(u_{\ell_n})}
    (\zeta)\right)\right|
    \xrightarrow[n\to\infty]{}0,
   \end{align*} by definition of $\ell_n$. The result now follows immediately
   by Slutsky's Theorem (see \cite[Theorem~6.3.15]{DM}).
\end{proof}

\section{Poisson Statistics via EVL}\label{sec:PoissonEVL}

As we have already mentioned, \cite[Theorem~5.2.1]{LLR83} states that for a stationary stochastic
process satisfying $D_1(u_n)$
and $D'(u_n)$, the EPP $N_n$ defined in \eqref{eq:def-EPP} converges in distribution
to a Poisson Process.

The main result in \cite{FF2} states that in order to prove an EVL for stationary
stochastic processes arising from a dynamical system, it suffices to
show conditions $D_2(u_n)$ and $D'(u_n)$. This proved to be an advantage over
\cite[Theorem~3.5.2]{LLR83} since the mixing information of systems is usually known through
decay of correlations that can be easily used to prove $D_2(u_n)$, as opposed to condition
$D_1(u_n)$ appearing in \cite[Theorem~3.5.2]{LLR83}.

Our goal here is to prove that we still get the Poisson limit if we
relax $D_1(u_n)$ so that it suffices to have sufficiently fast decay
of correlations of the dynamical systems that generate the
stochastic processes. However, for that purpose, one needs to
strengthen $D_2(u_n)$ in order to cope with multiple events.
(Something similar was necessary in the corresponding theory in
\cite{HSV}.)  For that reason we introduce condition $D_3(u_n)$
below, that still follows from sufficiently fast decay of
correlations, as $D_2(u_n)$ did, and together with $D'(u_n)$ allows
us to obtain the Poisson limit for the EPP.

Let $\S$ denote the semi-ring of subsets of $\R_0^+$ whose elements
are intervals of the type $[a,b)$, for $a,b\in\R_0^+$. Let $\RR$
denote the ring generated by $\S$. Recall that for every $A\in\RR$
there are $k\in\N$ and $k$ intervals $I_1,\ldots,I_k\in\S$ such that
$A=\cup_{i=1}^k I_j$. In order to fix notation, let
$a_j,b_j\in\R_0^+$ be such that $I_j=[a_j,b_j)\in\S$. For
$I=[a,b)\in\S$ and $\alpha\in \R$, we denote $\alpha I:=[\alpha
a,\alpha b)$ and $I+\alpha:=[a+\alpha,b+\alpha)$. Similarly, for
$A\in\RR$ define $\alpha A:=\alpha I_1\cup\cdots\cup \alpha I_k$ and
$A+\alpha:=(I_1+\alpha)\cup\cdots\cup (I_k+\alpha)$.

 For every $A\in\RR$ we define
\[
M(A):=\max\{X_i:i\in A\cap {\mathbb Z}\}.
\]
In the particular case where $A=[0,n)$ we simply write, as before,
$M_n=M[0,n).$

At this point, we propose:
\begin{condition}[$D_3(u_n)$]\label{cond:D^*} Let $A\in\RR$ and $t\in\N$.
We say that $D_3(u_n)$ holds for the sequence $X_0,X_1,\ldots$ if
\[ \mu\left(\{X_0>u_n\}\cap
  \{M(A+t
  )\leq u_n\}\right)-\mu(\{X_0>u_n\})
  \mu(\{M(A
  )\leq u_n\})\leq \gamma(n,t),
\]
where $\gamma(n,t)$ is nonincreasing in $t$ for each $n$ and
$n\gamma(n,t_n)\to0$ as $n\rightarrow\infty$ for some sequence
$t_n=o(n)$, which means that $t_n/n\to0$ as $n\to \infty$.
\end{condition}

Recalling the definition of the EPP $N_n(t)=N_n[0,t)$ given in
\eqref{eq:def-EPP}, we set $$N_n[a,b):=N(b)-N(a)=\sum_{j=\lceil
v_na\rceil}^{\lfloor v_nb\rfloor} \I_{\{X_j>u_n\}}.$$

We now state the main result of this section that gives the Poisson
statistics for the EPP under $D_3(u_n)$ and $D'(u_n)$.

\begin{theorem}\label{thm:D3+D'=>Poisson}
  Let $X_1, X_2,\ldots$ be a stationary stochastic process
  for which conditions $D_3(u_n)$ and $D'(u_n)$ hold for a
  sequence of levels $u_n$ such that $n\mu(X_0>u_n)
  \to\tau>0,$ as $n\to\infty$.
 Then the EPP $N_n$ defined
  in \eqref{eq:def-EPP} is such that
  $N_n\xrightarrow[]{d}N$, as $n\rightarrow\infty$, where $N$
  denotes a Poisson Process with intensity $1$.
\end{theorem}

As a consequence of this theorem, Theorem~\ref{thm:EPP=>HTPP} and
the results in \cite{FF} we get:

\begin{corollary}
  \label{cor:EPP-Poisson-BC} For any Benedicks-Carleson quadratic
  map $f_a$ (with $a\in\B$), consider a stochastic process
  $X_0, X_1,\ldots$ defined by \eqref{eq:def-stat-stoch-proc} and
  \eqref{eq:observable-form}, with $\zeta$ being either the critical
  point or the critical value. Then, denoting by $N_n$ the associated EPP
   as in \eqref{eq:def-EPP}, we have $N_n\xrightarrow[]{d}N$,
   as $n\rightarrow\infty$, where $N$
  denotes a Poisson Process with intensity $1$. Moreover, if we
  consider $N_n^*$, the HTPP as in \eqref{eq:def-HTPP},
  for balls around either
  the critical point or the critical value, then the same limit also
  applies to $N_n^*$.
\end{corollary}

With minor adjustments to \cite{Collexval}, we  can use
Theorem~\ref{thm:D3+D'=>Poisson} to show that, similarly to
Corollary~\ref{cor:EPP for acip}, interval maps with exponential
decay of correlations have Poisson statistics for the EPP.  However,
we will not state this result here, since we prove a more general
result (which works in higher dimensions) in
Section~\ref{sec:EVL-HTS-higher-dim}.

\iffalse With some minor adjustments, the results in
\cite{Collexval}  imply the following.  Since we actually prove this
result in the higher dimensional setting in
Section~\ref{sec:EVL-HTS-higher-dim}, we leave these adjustments to
the reader.

\begin{corollary}
  \label{cor:EPP-Poisson-Collet}Consider the system $(I,f,\mu)$ where $f$
  is non-uniformly hyperbolic $C^2$ map of the interval which admits an
 acip $\mu$,with exponential decay of correlations
 (just as in \cite{Collexval}). Consider a stochastic process
  $X_0, X_1,\ldots$ defined by \eqref{eq:def-stat-stoch-proc} and
  \eqref{eq:observable-form}, for some $\zeta$ in a full $\mu$-measure set
  given by \cite[Theorem~1.1]{Collexval}. Then, denoting by $N_n$ the associated EPP
   as in \eqref{eq:def-EPP}, we have $N_n\xrightarrow[]{d}N$,
   as $n\rightarrow\infty$, where $N$
  denotes a Poisson Process with intensity $1$. Moreover, if we
  consider $N_n^*$, the HTPP as in \eqref{eq:def-HTPP},
  for balls around $\zeta$, then the same limit also
  applies to $N_n^*$.
\end{corollary}

Collet noted that maps satisfying the above conditions are satisfied
for examples in \cite{Yorrr} and \cite{Yostat}.  Although in some
sense Corollary~\ref{cor:EPP for acip} is rather stronger than this,
in the sense that the maps there did not need to have an exponential
decay condition, this corollary is still useful for maps which are
not in $N\! F^2$. \textbf{MT: still need distortion!} \fi

\subsection{Proofs of the results}
 In this section we prove Theorem~\ref{thm:D3+D'=>Poisson} and Corollary~\ref{cor:EPP-Poisson-BC}. The key is
 Proposition~\ref{prop:indep-increments} whose proof we prepare with
 the following two Lemmas.  These are very similar to ones
 in \cite[Section~3]{FF2} and
 \cite[Section~3]{Collexval}, but we redo them here for completeness
 and because, in contrast to the original ones, we need them to take care
 of events that depend on nonconsecutive random variables.
\begin{lemma}
  \label{lem:prob-of-union}
  For any $\ell\in\N$ and $u\in\R$ we have
  \[
  \sum_{j=0}^{\ell-1} \mu(X_j> u)\geq \mu(M_\ell> u)
  \geq \sum_{j=0}^{\ell-1} \mu(X_j> u)-
  \sum_{j=0}^{\ell-1}\sum_{i=0,i\neq j}^{\ell-1} \mu(\{X_j> u\}\cap
  \{X_i> u\})
  \]
\end{lemma}
\begin{proof}
This is a straightforward consequence of the formula for the
probability of a multiple union of events. See for example the first
Theorem of Chapter~4 in \cite{Fe52}.
\end{proof}

\begin{lemma}
  \label{lem:relation-maximums_n}
  Assume that $r,s,\ell,t$ are nonnegative integers. Suppose that
  $A,B\in\RR$
  are such that $A\subset B$. Set
$\ell:=\#\{j\in\N:j\in B\setminus A\}$. Assume that $\min\{x:x\in
A\}\geq r+t$ and let $A_0=[0,r+t)$. Then, we have
  \begin{equation}\label{lem:relation-maximums-eq1_n}
  0\leq \mu(M(A)\leq u)-\mu(M(B)\leq u)\leq\ell \cdot \mu(X>u)
  \end{equation}
  and
  \begin{multline}\label{lem:relation-maximums-eq2_n}
  \left|\mu(M(A_0\cup A)\leq u)-\mu(M(A)\leq u)+
  \sum_{i=0}^{r-1} \mu\left(\{X>u\}\cap
  \{M(A-i)\leq u\}\right)\right|\\ \leq 2r\sum_{i=1}^{r-1}
  \mu(\{X>u\}\cap\{X_i>u\})+t\mu(X>u).
  \end{multline}
\end{lemma}

\begin{proof}
By the law of total probability and stationarity we have, for any $i\ge0$,
\begin{align*}
   \mu(M(A)\leq u)&= \mu(M(B)\leq u)+ \mu(\{M(A)\leq u\}\cap\{M(B\setminus A)>u\})
\\
&\le  \mu(M(B)\leq u)+ \mu(M(B\setminus A)>u)\\
&\leq \mu(M(B)\leq u)+\ell \mu(X>u)
\end{align*}
and the first statement of the Lemma follows.

 \noindent For the second statement observe that
$$
\{M(A_0\cup A)\leq u\}=\{M([0,r))\leq u\}\cap\{M([r,r+t))\leq
u\}\cap\{M(A)\leq u\}.
$$
Consequently,
$$
\left(\{M([0,r))\leq u\}\cap\{M(A)\leq u\}\right)\setminus
\{M(A_0\cup A)\leq u\}
 \subset \{M([r,r+t))> u\}.
$$
\noindent Thus, using the first inequality of
Lemma~\ref{lem:prob-of-union} we obtain
\begin{equation}
\label{eq:estimativa1} \big|\mu(\{M([0,r))\leq u\}\cap\{M(A)\leq
u\})-\mu(\{M(A_0\cup A)\leq u\}) \big| \le t
 \mu(X>u)\;.
\end{equation}
\noindent Using stationarity and the first inequality in
Lemma~\ref{lem:prob-of-union} we have
\begin{align*}
\mu(\{M([0,r))\leq u\}\cap\{M(A)\leq u\})&= \mu(\{M(A)\leq
u\})-\mu(\{M([0,r))> u\}\cap
\{M(A)\leq u\})\\
&\geq \mu(\{M(A)\leq u\})- \sum_{i=0}^{r-1} \mu(\{X_i>u\}\cap
  \{M(A)\leq u\}).
\end{align*}
Now, by the second inequality in Lemma~\ref{lem:prob-of-union} we
have
\begin{align*}
\mu(\{M([0,r))\leq u\}\cap\{M(A)\leq u\})&\leq \mu(\{M(A)\leq u\})-
\sum_{i=0}^{r-1} \mu(\{X_i>u\}\cap
  \{M(A)\leq u\})\\
  &\quad+\sum_{i=0}^{r-1}\sum_{\ell=0,i\neq\ell}^{r-1} \mu(\{X_i>u\}\cap
  \{X_\ell>u\}\cap \{M(A)\leq u\}).
\end{align*}
Finally, stationarity and the last three inequalities give
\begin{multline*}
\Big| \mu(\{M([0,r))\leq u\}\cap\{M(A)\leq u\})-
 \mu(\{M(A)\leq u\}) +\sum_{i=0}^{r-1} \mu(\{X>u\}\cap
  \{M(A-i)\leq u\})\Big|\\
\leq 2r\sum_{i=1}^{r-1} \mu(\{X>u\}\cap\{X_i>u\}),
\end{multline*}
 and the result follows by \eqref{eq:estimativa1}.
\end{proof}

\begin{proposition}\label{prop:indep-increments} Let $A\in\RR$ be \st
that $A=\bigcup_{j=1}^p I_j$ where $I_j=[a_j,b_j)\in\S$,
$j=1,\ldots,p$ and $a_1<b_1<a_2<\cdots<b_{p-1}<a_p<b_p$. Let
$\{u_n\}_{n\in\N}$ be such that $n\mu(X_0>u_n)\to\tau>0$, as
$n\to\infty$, for some $\tau\geq 0$. Assume that conditions
$D_3(u_n)$ and $D'(u_n)$ hold. Then,
\[
\mu\left(M\left(\n A\right)\leq
u_n\right)\xrightarrow[n\rightarrow+\infty]{}\prod_{j=1}^p \mu(M(\n
I_j)\leq u_n)=\prod_{j=1}^p \e^{-\tau(b_j-a_j)}.
\]
\end{proposition}

\begin{proof}
Let $h:=\inf_{j\in \{1,\ldots,p\}}\{b_j-a_j\}$ and
$H:=\lceil\sup\{x:x\in A\}\rceil$. Take $k>2/h$ and $n$ sufficiently
large. Note this guarantees that if we partition $n[0,H]\cap
{\mathbb Z}$ into blocks of length $r_n:=\lfloor n/k\rfloor$,
$J_1=[Hn-r_n,Hn)$, $J_2=[Hn-2r_n,Hn-r_n)$,\ldots,
$J_{Hk}=[Hn-Hkr_n,n-(Hk-1)r_n)$, $J_{Hk+1}=[0,Hn-Hkr_n)$, then there is
more than one of these blocks contained in $\n I_i$.  Let
$S_\ell=S_\ell(k)$ be the number of blocks $J_j$ contained in $\n
I_\ell$, that is,
$$S_\ell:=\#\{j\in \{1,\ldots,Hk\}:J_j\subset \n I_\ell\}.$$
As we have already observed $S_\ell>1$ $\forall \ell\in \{1,\ldots,p\}$.
For each $\ell\in \{1,\ldots,p\}$, we define
$$A_\ell:=\bigcup_{i=1}^{\ell}I_{p-i+1}.$$
Set $i_\ell:=\min\{j\in \{1,\ldots,k\}:J_j\subset \n I_{\ell}\}.$
Then $J_{i_\ell},J_{i_\ell+1},\ldots,J_{i_\ell+s_\ell}\subset n
I_\ell$. Now, fix $\ell$ and for each $ i\in \{i_{p-\ell
+1},\ldots,i_{p-\ell +1}+S_{p-\ell+1}\}$ let
$$
B_i:=\bigcup_{j=i_{p-\ell+1}}^i J_j,\;  J_i^*:=[Hn-ir_n,
Hn-(i-1)r_n-t_n)\; \mbox{ and } J_i':=J_i-J_i^*.$$ Note that
$|J_i^*|=r_n-t_n$ and $|J_i'|=t_n$. See Figure~\ref{fig:notation}
for more of an idea of the notation here.
\begin{figure}[h]
  \includegraphics[scale=1]{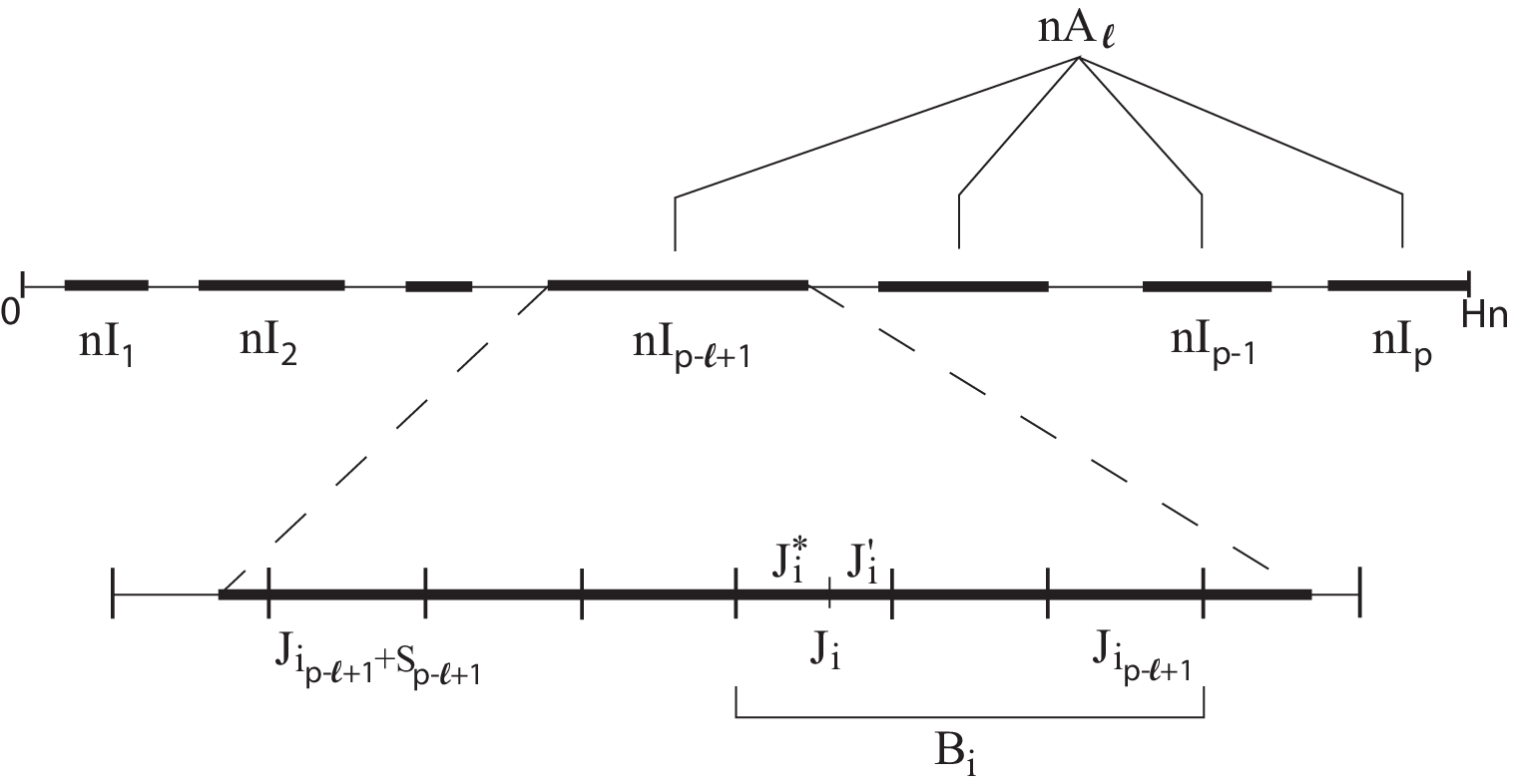} \caption{Notation}\label{fig:notation}
\end{figure}

We have,
\begin{align*}
|\mu(M(B_i\cup \n & A_{\ell-1})\leq
u_n)-(1-r_n\mu(X>u_n))\mu(M(B_{i-1}\cup \n A_{\ell-1})\leq u_n)|\\
&=\Big|\mu(M(B_i\cup \n A_{\ell-1})\leq u_n)-\mu(M(B_{i-1}\cup \n
A_{\ell-1})\leq u_n)\\&\qquad\qquad\quad+r_n \mu(X>
u_n)\mu(M(B_{i-1}\cup \n A_{\ell-1})
\leq u_n)\Big|\\
& \leq \Big|\mu(M(B_i\cup \n A_{\ell-1})\leq u_n)-\mu(M(B_{i-1}\cup
\n A_{\ell-1})\leq u_n)\\
&\qquad\qquad\quad+(r_n-t_n)\mu(X> u_n)\mu(M(B_{i-1}\cup \n A_{\ell-1}) \leq
u_n)\Big|\\
 &\quad+t_n\mu(X> u_n)\mu(M(B_{i-1}\cup \n A_{\ell-1}) \leq u_n)\\
&\leq \Big|\mu(M(B_i\cup \n A_{\ell-1})\leq u_n)-\mu(M(B_{i-1}\cup
\n A_{\ell-1})\leq u_n)\\
&\qquad\qquad\quad+\sum_{j=0}^{r_n-t_n-1}\mu(\{X_{j+Hn-ir_n}>u_n\}\cap
\{M(B_i\cup \n A_{\ell-1})\leq u_n)\}\Big|\\
&\quad+\Big|(r_n-t_n)\mu(X>u_n)\mu(M(B_{i-1}\cup \n A_{\ell-1})\leq
u_n)\\
&\qquad\qquad\quad-\sum_{j=0}^{r_n-t_n-1}\mu(\{X_{j+Hn-ir_n}>u_n\}\cap
\{M(B_i\cup \n A_{\ell-1})\leq u_n)\}\Big|\\&\quad+t_n\mu(X>u_n).
\end{align*}
By Lemma \ref{lem:relation-maximums_n}, we obtain
\begin{align*}
\Big|\mu(M(B_i\cup \n& A_{\ell-1})\leq u_n) -(1-r_n\mu(X>u_n))\mu\left(M(B_{i-1}\cup \n A_{\ell-1})\leq u_n\right)\Big|\\
&\leq 2(r_n-t_n)\sum_{j=1}^{r_n-t_n-1}\mu(\{X>u_n\}\cap\{X_j>u_n\})
+t_n\mu(X>u_n)\\
&\quad+ \sum_{j=0}^{r_n-t_n-1}\Big|\mu(X>u_n)\mu(M(B_{i-1}\cup \n
A_{\ell-1})\leq
u_n)\\
&\qquad\qquad\quad-\mu(\{X>u_n\}\cap\{M((B_{i-1}\cup \n
A_{\ell-1}) -d_j)\leq u_n\})\Big|\\
&\quad+t_n\mu(X>u_n),
\end{align*}
where $d_j=(j+Hn-ir_n)$. Now using condition $D_3(u_n)$, we obtain
\begin{multline*}
\Big|\mu(M(B_i\cup \n A_{\ell-1})\leq
u_n)-(1-r_n\mu(X>u_n))\mu(M(B_{i-1}\cup
\n A_{\ell-1})\leq u_n)\Big|\\
\leq 2(r_n-t_n)\sum_{j=1}^{r_n-t_n-1}\mu(\{X>u_n\}\cap\{X_j>u_n\})
+2t_n\mu(X>u_n)+(r_n-t_n)\gamma(n,t_n).
\end{multline*}
Set
$$\Upsilon_{k,n}:=2(r_n-t_n)\sum_{j=1}^{r_n-t_n-1}
\mu(\{X>u_n\}\cap\{X_j>u_n\})
+2t_n\mu(X>u_n)+(r_n-t_n)\gamma(n,t_n).$$
Recalling \eqref{eq:un}, we may assume that $n$ and $k$ are sufficiently large so that
$\frac{n}{k}\mu(X>u_n)<2$ and
$|1-r_n\mu(X>u_n)|<1$ which implies
\begin{align*}
&\left|\mu(M(B_{S_{p-\ell+1}}\cup \n A_{\ell-1})\leq
u_n)-(1-r_n\mu(X>u_n))\mu(M(B_{S_{p-\ell+1}-1}\cup \n A_{\ell-1})\leq
u_n)\right|\leq\Upsilon_{k,n},
\end{align*}
and
\begin{align*}
&\left|\mu(M(B_{S_{p-\ell+1}}\cup \n A_{\ell-1})\leq
u_n)-(1-r_n\mu(X>u_n))^2\mu(M(B_{S_{p-\ell+1}-2}\cup \n A_{\ell-1})\leq u_n)\right|\\
&\leq \Big|\mu(M(B_{S_{p-\ell+1}}\cup \n A_{\ell-1})\leq
u_n)-(1-r_n\mu(X>u_n))\mu(M(B_{S_{p-\ell+1}-1}\cup \n A_{\ell-1})\leq u_n)\Big|\\
&\quad+\left|1-r_n\mu(X>u_n)\right|\Big|\mu(M(B_{S_{p-\ell+1}-1}\cup
\n A_{\ell-1})\leq u_n)\\
&\hspace{7cm}-(1-r_n\mu(X>u_n))\mu(M(B_{S_{p-\ell+1}-2}\cup \n A_{\ell-1})\Big|\\
&\leq2\Upsilon_{k,n}.
\end{align*}
Inductively, we obtain
\begin{align*}
&\left|\mu(M(B_{S_{p-l+1}}\cup \n A_{\ell-1})\leq
u_n)-(1-r_n\mu(X>u_n))^{S_{p-\ell+1}}\mu(M(\n A_{\ell-1})\leq
u_n)\right|\leq S_{p-\ell+1}\Upsilon_{k,n}.
\end{align*}
Using Lemma \ref{lem:relation-maximums_n},
\begin{align*}
&\left|\mu(M(\n A_{\ell})\leq
u_n)-(1-r_n\mu(X>u_n))^{S_{p-\ell+1}}\mu(M(\n A_{\ell-1})\leq u_n)\right|\\
&\leq \left|\mu(M(\n A_{\ell})\leq u_n)-\mu(M(B_{S_{p-l+1}}\cup \n
A_{\ell-1})\leq
u_n)\right|\\
&\hspace{2cm}+\left|\mu(M(B_{S_{p-l+1}}\cup \n A_{\ell-1})\leq
u_n)-(1-r_n\mu(X>u_n))^{S_{p-\ell+1}}\mu(M(\n A_{\ell-1})\leq u_n)\right|\\
&\leq \left|\mu(M(\n I_{p-\ell+1}\cup \n A_{\ell-1})\leq
u_n)-\mu(M(\cup_{i=i_{\ell}}^{S_{p-\ell+1}}J_i\cup \n
A_{\ell-1})\leq
u_n)\right|+S_{p-l+1}\Upsilon_{k,n}\\
&\leq 2r_n \mu(X>u_n)+S_{p-l+1}\Upsilon_{k,n}.
\end{align*}
In the next step we have
\begin{align*}
&\Big|\mu(M(\n A_{\ell})\leq
u_n)-(1-r_n\mu(X>u_n))^{S_{p-\ell+1}+S_{p-\ell+2}}\mu(M(\n
A_{\ell-2})\leq
u_n)\Big|\\
&\leq \Big|\mu(M(\n A_{\ell})\leq
u_n)-(1-r_n\mu(X>u_n))^{S_{p-\ell+1}}\mu(M(\n A_{\ell-1})\leq u_n)\Big|\\
&\hspace{1cm}+\left|1-r_n\mu(X>u_n)\right|^{S_{p-\ell+1}}\Big|\mu(M(\n
A_{\ell-1})\leq
u_n)\\
&\hspace{7cm}-(1-r_n\mu(X>u_n))^{S_{p-\ell+2}}\mu(M(\n
A_{\ell-2})\leq
u_n)\Big|\\
&\leq 4r_n \mu(X>u_n)+(S_{p-\ell+1}+S_{p-\ell+2})\Upsilon_{k,n}.
\end{align*}
Therefore, by induction, we obtain
\begin{align*}
&\left|\mu(M(\n A_{p})\leq u_n)-(1-r_n\mu(X>u_n))^{\sum_{j=1}^p
S_{j}}\right|\leq 2pr_n \mu(X>u_n)+{\sum_{j=1}^{p}S_{j}}\Upsilon_{k,n}.
\end{align*}
Now, it is easy to see that $S_j\sim k|I_j|$, for each
$j\in\{1,\ldots,p\}$. Consequently,
\begin{align*}
&\lim_{k\rightarrow+\infty}\lim_{n\rightarrow+\infty}
\left(1-r_n\mu(X>u_n)\right)^{{\sum_{j=1}^p}S_{j}}=
\lim_{k\rightarrow+\infty} \lim_{n\rightarrow+\infty}
\left(1-\left\lfloor\frac{n}{k}\right\rfloor\mu(X>u_n)\right)^{{\sum_{j=1}^p}S_{j}}\\
&=\lim_{k\rightarrow+\infty}\left(1-\frac{\tau}{k}\right)
^{{\sum_{j=1}^p}S_{j}}=\lim_{k\rightarrow+\infty}
\left[\left(1-\frac{\tau}{k}\right)^{k\sum_{j=1}^p |I_j|}\right]
^{\frac{{\sum_{j=1}^p}S_{j}}{k\sum_{j=1}^p
|I_j|}}=\e^{-\tau \sum_{j=1}^p |I_j|}\\
&=\prod_{j=1}^p \e^{-\tau(b_j-a_j)}.
\end{align*}

To conclude the proof it suffices to show that
\[
\lim_{k\rightarrow+\infty}\lim_{n\rightarrow+\infty}(2pr_n
\mu(X>u_n)+ kH\Upsilon_{k,n})=0.
\]
We start by noting that, since $n\mu(X>u_n)\to\tau\geq 0$,
\[
\lim_{k\rightarrow+\infty}\lim_{n\rightarrow+\infty}2pr_n
\mu(X>u_n)= \lim_{k\rightarrow+\infty} \frac{2p\tau}{k}=0.
\]
Next we need to check that
\[
\lim_{k\rightarrow+\infty}\lim_{n\rightarrow+\infty}
k\Upsilon_{k,n}=0,
\]
which means,
\begin{multline*}
\lim_{k\rightarrow+\infty}\lim_{n\rightarrow+\infty}
2k(r_n-t_n)\sum_{j=1}^{r_n-t_n-1}\mu(\{X>u_n\}\cap\{X_j>u_n\})
+2kt_n\mu(X>u_n)\\+k(r_n-t_n)\gamma(n,t_n)=0.
\end{multline*}

Assume that $t=t_n$ where $t_n=o(n)$ is given by Condition
$D_3(u_n)$. Now, observe that, by \eqref{eq:un}, for every $k\in\N$,
we have \( \lim_{n\to\infty}kt_n\mu(X>u_n)=0 \). Finally, use
$D_3(u_n)$ and $D'(u_n)$ to prove that the two remaining terms also
go to $0$.
\end{proof}

\begin{proof}[Proof of Theorem~\ref{thm:D3+D'=>Poisson}]
  Since the Poisson Process has no fixed atoms, that is, points $t$
  such that $\mu\left(N(\{t\})>0\right)>0$, the convergence is
  equivalent to convergence of finite dimensional distributions.
  But, because $N$ is a simple point process, without multiple
  events, we may use a criterion proposed by Kallenberg
  \cite[Theorem~4.7]{Ka83} to show the stated convergence. Namely we
  need to verify that
\begin{enumerate}
\item $\E(N_n(I))\xrightarrow[n\to\infty]{}\E(N(I))$, for all $I\in
\mathcal S$;

\item $\mu(N_n(A)=0)\xrightarrow[n\to\infty]{}\mu(N(A)=0)$, for all
$A\in\mathcal R$,
\end{enumerate}
where $\E(\cdot)$ denotes the expectation with respect to $\mu$.

First we show that condition (1) holds. Let $a,b\in\R^+$ be such
that $I=[a,b)$, then, recalling that $v_n=1/\mu(X_0>u_n)$, we have
\begin{align*}
  \E(N_n(I))&=\E\left(\sum_{j=\lfloor v_na\rfloor+1}^{\lfloor v_nb\rfloor}
\I_{\{X_j>u_n\}}\right)=\sum_{j=\lfloor v_na\rfloor+1}^{\lfloor
v_nb\rfloor}E(\I_{\{X_j>u_n\}})\\
&=\left(\lfloor v_nb\rfloor-(\lfloor v_na\rfloor+1)\right)\mu(X_0>u_n)\\
&\sim (b-a)v_n\mu(X_0>u_n)\xrightarrow[n\to\infty]{}(b-a)=\E(N(I)).
\end{align*}

To prove condition (2), let $s\in\N$ and $A=\cup_{i=1}^s I_i$ where
$I_1,\ldots,I_s\in\mathcal S$ are disjoint. Also let
$a_j,b_j\in\R^+$ be such that $I_j=[a_j,b_j)$. By
Proposition~\ref{prop:indep-increments}, we have
\begin{equation*}
  \mu(N_n(A)=0)=\mu\left(\cap_{i=1}^s\{M(v_n I_j)\leq u_n\}\right)\sim
  \mu\left(\cap_{i=1}^s\{M((n/\tau) I_j)\leq u_n\}\right)
  \xrightarrow[n\to\infty]{}\prod_{j=1}^s\e^{-(b_j-a_j)}.
\end{equation*}
The result follows at once since
$\mu(N(A)=0)=\prod_{i=1}^s\mu(N(I_j)=0)=\prod_{j=1}^s\e^{-(b_j-a_j)}$.
\end{proof}

\begin{proof}[Proof of Corollary~\ref{cor:EPP-Poisson-BC}]
In \cite{FF,FF2}, conditions $D_2(u_n)$ and $D'(u_n)$ were proved for
stochastic processes $X_0,X_1,\ldots$ as in \eqref{eq:def-stat-stoch-proc} and
\eqref{eq:observable-form} with $\zeta$ being either the critical point or
the critical value and observables of type $g_3$ ($g_3(x)=x$ for $\zeta=1$
and $g_3=1-ax^2$ for $\zeta=0$).

Observe that independently of the type of
$g$, the sequence $u_n$ is computed so that an exceedance of the level
$u_n$ corresponds to
a visit to the ball $B_{\delta_n}(\zeta)$, where $\delta_n$ is such that
$\mu(B_{\delta_n})\sim\tau/n$. This means that condition $D'(u_n)$ can be
written in terms of returns to $B_{\delta_n}(\zeta)$ which implies that it
holds for every sequence $X_0,X_1,\ldots$, independently of the shape of $g$.

Condition $D_3(u_n)$ follows from decay of correlations. In fact,
from \cite{KN92,Yo92} one has that for all
$\phi,\psi:M\rightarrow\R$ with bounded variation, there is
$C,\alpha>0$ independent of $\phi, \psi$ and $n$ such that
\begin{equation}
\label{eq:decay-correlations} \left| \int\phi\cdot(\psi\circ
f^t)d\mu-\int\phi d\mu\int\psi d\mu\right|\leq
C\mbox{Var}(\phi)\|\psi\|_\infty \e^{-\alpha t},\quad\forall
t\in\N_0,\end{equation} where $\mbox{Var}(\phi)$ denotes the total
variation of $\phi$. For each $A\in\RR$, take
$\phi=\I_{\{X_0>u_n\}}$ and $\psi=\I_{\{M(A)\leq u_n\}}$, then
\eqref{eq:decay-correlations} implies that Condition~$D_3(u_n)$
holds with $\gamma(n,t)=\gamma(t):=2C\e^{-\alpha t}$ and for the
sequence $t_n=\sqrt n$, for example.
\end{proof}

\iffalse
\begin{proof}[Proof of Corollary~\ref{cor:EPP-Poisson-Collet}]
In \cite{Collexval}, conditions $D_2(u_n)$ and $D'(u_n)$ were proved
for stochastic processes $X_0,X_1,\ldots$ as in
\eqref{eq:def-stat-stoch-proc} and \eqref{eq:observable-form}  for
any $\zeta$ in a  full $\mu$-measure set of points and observables
of type $g_1$ ($g_1(x)=-\log x$). The same argument used in the
previous proof would allow us to conclude that condition $D'(u_n)$
holds for every sequence $X_0,X_1,\ldots$, independently of the
shape of $g$. Condition $D_3(u_n)$ also follows from decay of
correlations using almost the same argument as before. The
difference would be that, in this setting, decay of correlations is
available for H\"older continuous functions against $L^\infty$ ones,
instead. This means that we can not use the test function
$\phi=\I_{\{X_0>u_n\}}$, as we did previously. However, using a
suitable H\"older approximations the same result follows. See
\cite[Lemma~3.3]{Collexval}.\end{proof}
\fi
\section{Poisson Statistics for first return times}

\label{sec:Poisson RTS}
\iffalse
\textbf{MT:I'm still a little uneasy about the Poisson Process versus distribution here.  Is the following true?  (I mean, why does the first displayed equation only imply a Poisson distribution, when it looks the same as stuff in the previous sections? }
\fi

The purpose of this section is to discuss what is known about the
Poisson statistics of first return times to balls.  The main focus
is on showing that a map $f\in N\!F^2$ with an acip must have the
RTPP asymptotically converging to a Poisson Process.  However, for
more generality we will introduce the ideas assuming that our phase
space $\X$ is a Riemannian manifold. We note that a similar result
to the main theorem \cite{HLV} implies that the limit laws for the
HTPP and RTPP are the same. So since the results we will cite below
are usually given in terms of RTS, we will use this.

Similarly to the proof of Theorem~\ref{thm:D3+D'=>Poisson}, in order
to show that the RTPP has a Poisson limit, it
suffices to prove that for $k\in \N$ and a rectangle $R_k\subset
\R^k$,
$$\left|\mu_{U_n}\left((w_{U_n}, w_{U_n}^2, \ldots, w_{U_n}^
k)\in
\frac1{\mu(U_n)}R_k\right) - \int_{R_k}\Pi_{i=1}^k \e^{t_i}~dt^k\right|
 \to 0 \text{ as } n\to \infty.$$

The main result of \cite{BSTV} is that the RTS for an inducing
scheme is the same for the inducing scheme as for the original
system. However, they remark in that paper that their methods extend
to give the same Poisson statistics for the inducing scheme and the
original system.  In \cite{BTintstat}, the theory in \cite{BSTV} was
extended to show that for multimodal maps of the interval the RTS of
suitable inducing schemes converge to the RTS of the original
system.  The corresponding result for Poisson statistics follows
similarly.  For multimodal maps $f:I\to I$, with an acip $\mu$,
those inducing schemes are Rychlik maps.  Therefore to prove that
the original $(I,f,\mu)$ has the RTPP converging to a Poisson
process, we must show that the induced, Rychlik, maps also have this
property. As we sketch below, this can be proved using \cite[Theorem
2.6]{HSV}.

For a system $(X,F,\mu)$, we say that a partition ${\mathcal Q}$ is
\emph{uniform mixing} if there exists $\gamma_{\mathcal Q}(n)\to 0$
as $n\to \infty$, such that
$$\gamma_{\mathcal Q}(n):=\sup_{k,l}\sup_{\stackrel{A\in \sigma{\mathcal Q}_k}
{B\in F^{-(n+k)}\sigma{\mathcal Q}_l}}\left|\mu(A\cap B)
-\mu(A)\mu(B)\right|.$$  Here ${\mathcal
Q}_k:=\bigvee_{j=0}^{k-1}F^{-j}{\mathcal Q}$ and $\sigma {\mathcal
Q}_k$ is the sigma algebra generated by ${\mathcal Q}_k$. For our
purposes ${\mathcal Q}$ will be $\{U, U^c\}$ where $U$ is a  ball
around $\zeta$.

By \cite[Theorem 2.6]{HSV}, if we assume the system is uniform mixing for $\{U,U^c\}$, then for a rectangle $R_k\subset \R^k$,
\begin{equation}\left|\mu_U\left((w_U, w_U^2, \ldots, w_U^k)\in
\frac1{\mu(U)}R_k\right) - \int_{R_k}\Pi_{i=1}^k
\e^{t_i}~dt^k\right| \le Err(k, U).\label{eq:RTS pois}\end{equation}
Moreover, the term  $Err(k,U)$ goes to 0 as $U$ shrinks to a point
$\zeta$. In fact, we have $Err(k,U)=k\left(3d(U)+R(k, U)\right)$
where $R(k, U)\to 0$ as $\mu(U) \to 0$ and the rate that $R(k, U)$
goes to zero depends on how $\gamma_{\mathcal Q}$ shrinks with $U$ .
As was shown in \cite{BSTV}, for Rychlik maps the quantity $d(U)$
tends 0 as $U\to \{\zeta\}$.   Therefore it only remains to show
that the Rychlik maps defined in \cite{BTintstat} are uniform mixing
for $\{U, U^c\}$.

Since we assumed that $(X,F,\mu)$ is Rychlik, \cite[Theorem 5]{Rych} implies that the natural partition $\mathcal P_1$, consisting of maximal intervals on which $f$ is a homeomorphism, is Bernoulli, with exponential speed.  Since $(X,F, \mu)$ is uniformly expanding, this implies that $\{U, U^c\}$ is also Bernoulli, with exponential speed. As noted in \cite[Remark 2.5]{HSV}, this implies that $\{U, U^c\}$ is uniform mixing, as required.

The proof that the successive returns form a point process
converging to a Poisson Process follows from \eqref{eq:RTS pois} and
the Kallenberg argument used in the proof of
Theorem~\ref{thm:D3+D'=>Poisson}.

\section{EVL and HTS in higher dimensions}
\label{sec:EVL-HTS-higher-dim}

In this section, we extend Collet's theory of maps with exponential decay of correlations from one dimension to higher dimensions.  We conclude with an example.

Let $\X$ be as usual a $d$-dimensional compact Riemannian manifold
and $f:\X\to\X$ a $C^2$ endomorphism.   We say that $f$ \emph{admits
a Young tower} if there exists a ball  \( \Delta \subset \X \), a
countable partition \( \mathcal P \) (mod 0) of \(
    \Delta \) into topological balls \( \Delta_i \) with smooth boundaries,
    and a return time function \( R: \Delta \to \mathbb N \) piecewise
    constant on elements of  \(\mathcal P \) satisfying the following
    properties:

\renewcommand{\theenumi}{\textbf{Y$_{\arabic{enumi}}$}}

\begin{enumerate}
    \item
    {\em {Markov}:} \label{Y-Markov} for each \( \Delta_i\in\mathcal P \) and
    \( R=R(\Delta_i)
    \),
    \(
    f^{R}: \Delta_i \to \Delta
    \)
    is a \( C^{2} \) diffeomorphism (and in particular a bijection).  Thus the
    induced map
    \[ F: \Delta \to \Delta \ \text{  given by } \
     F(x) = f^{R(x)}(x)
     \] is
    defined almost everywhere and satisfies the classical Markov property.
    We consider also the \emph{separation time} \( s(x,y) \) given by the maximum
    integer such that \( F^{i}(x) \) and \( F^{i}(y) \) belong to the same
    element of the partition \( \mathcal P \) for all \( i\leq s(x,y) \),
    which we assume to be defined and finite for almost every pair of
    points \( x,y\in\Delta \).

    \item {\emph{Uniform backward contraction}:}
    \label{Y-uniform-back-contraction}
    There exist $C>0$ and
    \(0< \beta< 1 \) such that
    for \( x,y\in \Delta \) and any $0\leq n\leq s(x,y)$ we have
    \[ \dist(f^n(x),f^n(y))\leq C\beta^{s(x,y)-n}. \]

    \item {\em {Bounded distortion}:}\label{Y-bounded-distortion}
    For any $x,y\in\Delta$ and any $0\leq k\leq n< s(x,y)$ we have
    \[
   \log \prod_{i=k}^n\frac{\det Df(f^i(x))}
   {\det Df(f^i(x))}\leq C\beta^{s(x,y)-n}
    \]

    \item \label{Y-integrable-return}
    {\em {
    Integrable return times}:}
\[
\int R\; d\l<\infty
\]
\end{enumerate}

In this section we only consider maps admitting a Young tower
\emph{with exponential return time tail} which means that we will
replace condition \eqref{Y-integrable-return} by the following
stronger one
\renewcommand{\theenumi}{\textbf{Y$_{\arabic{enumi}}$'}}
\begin{enumerate}
\setcounter{enumi}{3}
    \item {\em {Exponential tail decay}:} \label{Y-exponential-tail}
    There is $C,\alpha>0$ such that
  \[ \l(\{R>n\})= C\e^{-\alpha n}.  \]
\end{enumerate}

\renewcommand{\theenumi}{\arabic{enumi}}

These systems have been studied, in a more general context, by L.S.
Young in \cite{Yostat,Yorrr}, where several examples can also be
found. Among the properties proved by L.S. Young we mention the
existence of an $F$-invariant measure $\mu_0$ that is equivalent to
Lebesgue measure on $\Delta$ (meaning that its density is bounded
above and below by a constant). After saturating one gets an
absolutely continuous (w.r.t. Lebesgue), $f$-invariant probability
given by
\begin{equation}
\label{eq:def-saturation}
\mu(A)=\bar R^{-1}\sum_{\ell=0}^\infty
\mu_0\left(f^{-\ell}(A)\cap\{R>\ell\}\right),
\end{equation}
where $\bar R=\int_\Delta R d\mu_0$. One of the main achievements in
\cite{Yostat,Yorrr} is the fact that the decay of the tail of return
times determines the speed of decay of correlations for H\"older
continuous (or Lipschitz) observables. Namely, if $\phi:\X\to \R$ is H\"older
continuous of exponent $0<\iota\leq1$, with \emph{H\"older constant}
\[K_\iota(\phi):=\sup_{x\neq
y}\frac{|\phi(x)-\phi(y)|}{\left(\dist(x,y)\right)^\iota}, \]
$\psi:\X\to \R$ is in $L^\infty(\l)$ and the tower has exponential
tail, then there are $C>0$ and $\alpha'>0$ such that
\begin{equation}
\label{eq:decay-correlations-Holder} \left| \int\phi\cdot(\psi\circ
f^t)d\mu-\int\phi d\mu\int\psi d\mu\right|\leq
CK_\iota(\phi)\|\psi\|_\infty \e^{-\alpha t},\quad\forall
t\in\N_0.\end{equation}

\begin{theorem}
  \label{thm:Collet-multi-dimensional}Let $\X$ be a $d$-dimensional
  compact Riemannian manifold and assume that $f:\X\to\X$ is a $C^2$ endomorphism
  admitting a Young tower with exponential tail. Consider a stochastic
  process $X_0, X_1,\ldots$ defined by \eqref{eq:def-stat-stoch-proc} and
  \eqref{eq:observable-form}, for some choice of $\zeta\in\X$. Then, for
  $\l$-almost every
  $\zeta\in\X$ chosen, conditions $D_3(u_n)$ (or $D_2(u_n)$) and
  $D'(u_n)$ hold, where $u_n$ is a sequence of levels satisfying
  \eqref{eq:un}.
\end{theorem}

Together with the results in Section~\ref{subsec:main results},
Section~\ref{sec:link-HTPP-EPP} and Section~\ref{sec:PoissonEVL} we
get the following corollary.

\begin{corollary}
  \label{cor:HTS-Collet-higher-dim}
Let $\X$ be a $d$-dimensional
  compact Riemannian manifold and assume that $f:\X\to\X$ is a $C^2$ endomorphism
  admitting a Young tower with exponential tail. Consider a stochastic
  process $X_0, X_1,\ldots$ defined by \eqref{eq:def-stat-stoch-proc} and
  \eqref{eq:observable-form} for some $\zeta\in\X$.
  Then, for $\l$-almost every choice of $\zeta\in\X$,
  the following assertions hold:
\begin{enumerate}
     \item We have an EVL for $M_n$, defined in
     \eqref{eq:def-max}, which coincides with that one of $\hat M_n$
     defined in \eqref{eq:def-max-iid}. In particular, it must be of
     one of the three classical types. Moreover,
for every $i\in\{1,2,3\}$, if $g$ is of type $g_i$ then we have an
EVL for $M_n$ of type $\ev_i$.

    \item We have
  exponential HTS to balls at $\zeta\in\X$.

  \item The EPP $N_n$ defined
  in \eqref{eq:def-EPP} is such that
  $N_n\xrightarrow[]{d}N$, as $n\rightarrow\infty$, where $N$
  denotes a Poisson Process with intensity $1$.

    \item The same applies to the HTPP $N_n^*$ defined in
    \eqref{eq:def-HTPP}.

\end{enumerate}

\end{corollary}

\subsection{Proof of Theorem~\ref{thm:Collet-multi-dimensional}}
\label{subsec:proof-thm-collet-multi-dim} To show this result, one
needs only to realise that Collet's proof of
\cite[Theorem~1]{Collexval} may be mimicked in our multi-dimensional
setting with minor adjustments. Thus, instead of repeating all the
arguments, we will prove that $D_3(u_n)$ and $D'(u_n)$ hold just by
redoing the parts that need to be adapted to this more general
higher dimensional setting. However, in order to keep on track we
will restate all the Lemmas (with the necessary adjustments) of
Collet's proof.

The first lemma is technical and very simple to prove.

\begin{lemma}
For any $\v \geq 1$
$$
\sum_{l,\,R_{l}>\v }R_{l}\;\l(\Delta_{l})
\le\begin{cases} 2\sum_{s=\v /2}^{\infty}\;\l(\{R>s\})\;,&\\
\sum_{s=\v }^{\infty}\;\l(\{R>s\})+\v \,\l(\{R>\v \})\;.&
\end{cases}
$$
\end{lemma}

See \cite[Lemma~2.1]{Collexval}.

Next result gives a relation between the measure $\mu$ of small sets
and their respective Lebesgue measure.

\begin{lemma}
\label{lem:collet-multidim-relation-mu-Leb} Under
\eqref{Y-exponential-tail}, there are two positive constants $C$ and
$\theta$ such that for any Lebesgue measurable set $A$,  we have
$$
\mu(A)\le C\l(A)^{\theta}.
$$
\end{lemma}

See \cite[Lemma~2.2]{Collexval}.

To prove condition $D'(u_n)$ we need to show that the $\mu$-measure
of the set of points $x$ that are too rapidly recurrent is small.
For every $\v \in\N$ and any $\epsilon>0$ we define the set $\eveps
$ of points that come back very close to the initial position after
$\v $ iterates
$$
\eveps =\{x\,,\,|x-f^{\v }(x)|<\epsilon\}\;.
$$ Since  $\X$ is compact and $f$
is $C^2$ we may define
$\Upsilon:=\sup\{\|Df(x)\|_\infty\:\;x\in\X\}$.

\begin{propositionP}\label{prop:main-estimate-collet-multidim}
Under \eqref{Y-exponential-tail}, there exist positive constants
$C$, $\alpha'$ and $\eta<1$ such that for any integer $\v $ and any
$\epsilon>0$ we have
$$
\mu(\eveps )\le C\left(\v ^{2}\epsilon^{\eta}+e^{-\alpha' \v
}\right).
$$
\end{propositionP}

\begin{proof} We follow the proof of the corresponding result
\cite[Proposition~2.3]{Collexval} very closely.
 We will  consider the intersection with  $\eveps $ of the various
 cylinders where $f^{\v }$ is one-to-one.
From \eqref{eq:def-saturation}, we have to consider the intersection
of these sets with $f^{j}(\Delta_{l})$. We will start by choosing a
number $1/2>\xi>0$ such that $\beta \Upsilon^{\xi}<1$ and assume
first that $R_{l}<\xi \v $. If we apply $f^{R_{l}-j}$ on
$f^j(\Delta_l)$, we land in $\Delta$ and we have to apply $f^{\v
-R_{l}+j}$ to get the image under $f^{\v}$. At this point it is
convenient to introduce the following construction. Let $(s_{j})$ be
a sequence of integers. We denote by $\Delta_{s_{1},
s_{2},\ldots,s_{r}}$ the set
$$
\Delta_{s_{1}, s_{2},\ldots,s_{r}}=\Delta_{s_{1}}\cap
f^{-R_{s_{1}}}\Delta_{s_{2}}\cap
f^{-(R_{s_{1}}+R_{s_{2}})}\Delta_{s_{3}} \cap\;\cdots\;\cap
f^{-(R_{s_{1}}+\,\cdots\,+R_{s_{r-1}})}\Delta_{s_{r}}\;.
$$

In other words, this is the subset $A$ of $\Delta_{s_{1}}$ which is
mapped by $f^{R_{s_{1}}+\cdots+R_{s_{r-1}}}$ bijectively onto
$\Delta_{s_{r}}$ with
$$
f^{R_{s_{1}}+\cdots+R_{s_{p}}}(A)\subset \Delta_{s_{p+1}}
$$
for $p=1,\ldots,\;r-1$.

For  fixed $\v $, $l$ and $j$,   we now consider all the sets
$\Delta_{s_{1},\ldots,s_{r}}$ with
$R_{s_{1}}+\cdots+R_{s_{r-1}}+R_{l}-j<\v $ and
$R_{s_{1}}+\cdots+R_{s_{r}}+R_{l}-j\ge \v $. Together with
$\{R>\v -1-R_{l}+j\}$, this gives  a partition of $\Delta$.
 We then construct
a partition of $f^{j}(\Delta_{l})$ by pulling back this partition by
$f^{R_{l}-j}$. We now consider $f^{\v }$ on each atom of this
partition. Let
$$
A=A_{l,j,s_{1},\ldots,s_{r}}=f^{j}(\Delta_{l})\cap
f^{j-R_{l}}\left(\Delta_{s_{1},\ldots,s_{r}}\right).
$$

We first assume that $R_{s_{r}}<\xi \v $ and $A$ has a `large' image under
$f^{\v }$, namely
$$
|f^{\v }(A)|\ge\delta\;,
$$
where $\delta$ is a positive number to be chosen adequately later
on. Let $J:=R_{s_{1}}+\cdots+R_{s_{r}}+R_{l}-j$ and $B:={\eveps
}\cap A$, which we may assume to be nonempty. We argue that $|f^{\v
}(B)|= O(\epsilon^d)$. To see this, let $x,y\in A$ be such that
$$\dist(f^\v (x),f^\v (y))=\mbox{diam}(f^\v (B)):=\sup\{\dist(z,w):\;
z,w\in f^\v (B)\}.$$

By \eqref{Y-uniform-back-contraction}, it follows that
\begin{equation*}
  \dist(x,y)\leq C \beta^{J} \dist(f^J(x),f^J(y)).
\end{equation*}
Moreover, by definition of $\Upsilon$, we also have
\begin{equation*}
\dist(f^\v (x),f^\v (y))\geq \Upsilon^{-J+\v }
\dist(f^J(x),f^J(x))\geq \Upsilon^{-R_{s_r}} \dist(f^J(x),f^J(x)).
\end{equation*}
Consequently, since by assumption $R_{s_{r}}<\xi \v $, we have
\begin{equation*}
\frac{\dist(x,y)}{\dist(f^\v (x),f^\v (y))}\leq C\beta
^{J}\Upsilon^{R_{s_r}}\leq \O(\beta ^{\v }\Upsilon^{\xi \v })\leq
\O(\left(\beta \Upsilon^{\xi}\right)^{\v }).
\end{equation*}
Since $\beta \Upsilon^{\xi}<1$, by the choice of $\xi$, then we may
pick $\v _0\in\N$, only depending on $f$, such that for all $\v \geq
\v _0$ we have
\begin{equation}
\label{eq:auxliary-computation}
 \dist(x,y)\leq
\frac12\,\dist(f^\v (x),f^\v (y)).
\end{equation}
Observe that for $\v <\v _0$,
Proposition~\ref{prop:main-estimate-collet-multidim} simply holds
with $C$ sufficiently large.

Now, assuming that $x,y\in B\subset \eveps $ we have
\begin{align*}
  \dist(f^\v (x),f^\v (y))&\leq
  \dist(f^\v (x),x)+\dist(x,y)+\dist(y,f^\v (y))\\
  &\leq \epsilon+ \frac12\,\dist(f^\v (x),f^\v (y))+\epsilon.
\end{align*}
This means that $\mbox{diam}(f^\v (B))\leq 4\epsilon$. If $x,y\notin
B$ then we could replace them by close enough $x',y'\in B$ so that
$\mbox{diam}(f^\v (B))\leq 5\epsilon$. Hence, we have proved that
$\left|f^\v (B)\right|=O(\epsilon^d)$.

Using distortion, we get
$$
|B|/|A|=O(\epsilon^d/\delta),
$$
and
$$
|\Delta_l\cap f^{-j}(B)|/|\Delta_l\cap f^{-j}(A)|=O\left(\epsilon^d/\delta\right).
$$

Since $\mu_{0}$ is equivalent to the Lebesgue measure on $\Delta$,
then
$$
\mu_{0}(\Delta_l\cap f^{-j}(B))=O\left({\frac{\epsilon^d}{
\delta}}\right)\mu_{0}(\Delta_l\cap f^{-j}(A)).
$$
Next we sum over all sets $A$ as above, contained in
 $f^{j}(\Delta_{l})$ and such that they have `large' image under $f^\v $.
 Since they are disjoint
 we get a contribution bounded above by
$O(\epsilon^d/\delta)\mu_{0}(\Delta_{l})$. Summing over $j$ we
get a bound $O(\epsilon^d/\delta)R_{l} \mu_{0}(\Delta_{l})$.
Finally, summing over $l$ we get the estimate:
$O(\epsilon^d/\delta)$. This ends the estimate in the good case
when segments $A$ reach `large' scale in $\v $ steps.

We next have
to gather the estimates for all the left-over bad cases. These bad
cases are dealt with by realising that they correspond to large
values of $R$, whose tail we are assuming to decay exponentially
fast. We skip the study of these cases and refer the reader to
\cite[Proposition~2.3]{Collexval} where they are treated without any
particular unidimensional argument.

Finally, collecting all the estimates, there exists $C>0$ so that
\begin{multline*}
\mu({\eveps })= O\Big({\frac{\epsilon^d}{\delta}}+\sum_{s>\xi \v
/2}\mu_{0}(R>s)\\+\v  \l(R\geq(1-\xi) \v )+\v ^{2}\mu(R\geq\xi \v )
+\v ^{2}\mu\big(R>C\log\delta^{-1}\big)\Big).
\end{multline*}
Using Lemma~\ref{lem:collet-multidim-relation-mu-Leb} and
\eqref{Y-exponential-tail} we have
$$
\mu({\eveps })=O\left({\frac{\epsilon^d}\delta}+\v^2
e^{-\alpha\theta\xi \v }+\v ^{2}\delta^{\gamma}\right)
$$
for some $1>\gamma>0$. The result follows by taking the minimum with
respect to $\delta$.
\end{proof}

Let $\{E_{\v }\}_{\v\in \N}$ be the sequence of sets defined by
$$
E_{\v }=\left\{y:\; \exists j \in\{1,\ldots,(\log \v )^{5}\},\,
|y-f^{j}(y)|\le \v ^{-1}\right\}\;.
$$

\begin{corollaryP}
There exist positive constants $C'$ and $\beta'<1$ such that for any
integer $\v $
$$
\mu(E_{\v })\le C' \,\v ^{-\beta'}\;.
$$
\end{corollaryP}

See \cite[Corollary~2.4]{Collexval}.

We not only need to control the set of points which recur too fast,
but also the set of points for which a neighbour recurs too fast.
For positive numbers $\psi$ and $\rho$ to be fixed below, we define
a sequence of measurable sets $\{F_{\v }\}_{\v\in \N}$  by
$$
F_{\v }=\left\{x:\; \mu\big(B_{\v ^{-\psi}}(x)\cap E_{\v
^{\psi}})\ge \kappa\,\v ^{-(d+\rho)\psi}\right\}\;.
$$

\begin{lemma}\label{lem:bad-set}
There exist positive numbers $\rho$ and $\psi$ such that
$\l\left(\bigcap_{i\geq1}\bigcup_{\v\geq i}F_\v\right)=0.$
\end{lemma}

We refer to \cite[Lemma~2.5]{Collexval} and references therein. The
proof uses maximal functions and a result by Hardy and Littlewood
which still holds in higher dimensions.

As we have seen in the proof of Corollary~\ref{cor:EPP-Poisson-BC},
it is very easy to show that $D_3(u_n)$ holds when we have decay of
correlations for observables of bounded variation. However, in this
setting, decay of correlations is only available for H\"older
continuous functions against $L^\infty$ ones, instead (see
\eqref{eq:decay-correlations-Holder}). This means that we cannot use
the test function $\phi=\I_{\{X_0>u_n\}}$, as we did before.
However, proceeding as in \cite[Lemma~3.3]{Collexval}, we use a
suitable H\"older approximation and show that the same result
follows:

\begin{lemma}
\label{lem:decay-correl-Dun} Assume that there exists a rate
function $\Theta:\N\to \R$, such that for every H\"older continuous
(or Lipschitz) observable $\phi$ and all $L^\infty$ observable
$\psi$ we have:
\[
\left| \int\phi\cdot(\psi\circ f^t)d\mu-\int\phi d\mu\int\psi
d\mu\right|\leq K_\iota(\phi)\|\psi\|_\infty \Theta(t),\quad\forall
t\in\N_0.
\]
Then, for every $\zeta\in\X$, $0<s<1$, $\eta>0$ and all measurable
set W we have
\[
\left|\mu(B_s(\zeta)\cap f^{-t}(W))-\mu(B_s(\zeta))\mu(W)\right| \le
s^{-(1+\eta)}\Theta(t)+O(s^{\theta(d+\eta)}),
\]
where $\theta$ is the number given in
Lemma~\ref{lem:collet-multidim-relation-mu-Leb}.
\end{lemma}

\begin{proof}
For a fixed $\eta>0$ we build the H\"older approximation $\phi$
of $\I_{B_s(\zeta)}$. Let $B:=B_s(\zeta)$ and $D:=\overline{B_{s-s^{1+\eta}}(\zeta)}$,
 where $\bar A$ denotes the closure of $A$.
Define $\phi:\X\to\R$ as
\[
\phi(x)=\begin{cases}
  0&\text{if $x\notin B$}\\
  \frac{\dist(x,\X\setminus B)}{\dist(x,\X\setminus B)+
  \dist(x,D)}&
  \text{if $x\in B\setminus D$}\\
  1& \text{if $x\in D$}
\end{cases}.
\]
Observe that $\phi$ is H\"older  continuous (Lipschitz) with
H\"older constant $s^{-(1+\eta)}$.

Now, we apply the decay of correlations to the H\"older continuous
function $\phi$ against $\I_{W}\in L^{\infty}$ to get
\[
\left| \int\phi\cdot(\I_{W}\circ f^t)d\mu-\int\phi d\mu\int\I_{W}
d\mu\right|\leq s^{-(1+\eta)}\Theta(t).
\]
Noticing that the support of $\I_B-\phi$ is contained in $B\setminus
D$ whose Lebesgue measure is $O(s^{d+\eta})$ and using
Lemma~\ref{lem:collet-multidim-relation-mu-Leb} we get
\[
\left|\mu(B_s(\zeta)\cap f^{-t}(W))-\mu(B_s(\zeta))\mu(W)\right| \le
s^{-(1+\eta)}\Theta(t)+O(s^{\theta(d+\eta)}).
\]
\end{proof}

\begin{proof}[Proof of Theorem~\ref{thm:Collet-multi-dimensional}]
First let us show that $D_3(u_n)$ holds. Since in this setting we
have exponential decay of correlations for H\"older continuous
functions (see \eqref{eq:decay-correlations-Holder}) and
$\{X_0>u_n\}=B_{g^{-1}(u_n)}(\zeta)$ then by
Lemma~\ref{lem:decay-correl-Dun} we may take
\[
\gamma(n,t)=O\Big((g^{-1}(u_n))^{-1-\eta}\e^{-\alpha
t}+(g^{-1}(u_n))^{\theta(d+\eta)}\Big).
\]
Hence, recalling that
$g^{-1}(u_n)\sim\left(\frac{\tau}{\kappa\rho(\zeta)n}\right)^{1/d}$,
if we consider $t_n=\sqrt n$, for example, and choose $\eta$ from
Lemma~\ref{lem:decay-correl-Dun} so that $\theta(d+\eta)/d>2$ (where
$\theta$ is given by
Lemma~\ref{lem:collet-multidim-relation-mu-Leb}), then we easily get
that $n\gamma(n,t_n)\xrightarrow[n\to\infty]{} 0$ which gives
$D_3(u_n)$.

Now, it only remains to show that $D'(u_n)$ also holds. Recall
the stochastic process $X_0,X_1,\ldots$ given by
\eqref{eq:def-stat-stoch-proc} for observables defined by
  \eqref{eq:observable-form}, achieving a global maximum at
  $\zeta\in\X$. At this point, we describe the full Lebesgue measure
  set of points $\zeta\in\X$ for which
  Theorem~\ref{thm:Collet-multi-dimensional} holds. We take $\zeta$
  for which
Lebesgue's differentiation theorem holds (with respect to the measure
$\mu$) and $\zeta\in \cup_{i\geq1}\cap_{j\geq i} \X\setminus F_j$,
which by Lemma~\ref{lem:bad-set} is also a full Lebesgue measure
set. For each such $\zeta$, let $\v_0(\zeta)\in\N$ be such that
$\zeta\notin F_j$ for all $j\geq \v_0(\zeta)$.

We consider a turning instant $t=t(n)=\lfloor(\log n)^2\rfloor$, and
split the sum in $D'(u_n)$ into the period before $t$ and after $t$.

For the later we use exponential decay of correlations
\eqref{eq:decay-correlations-Holder} and
Lemma~\ref{lem:decay-correl-Dun} to get, for some $C>0$,
\begin{align*}
S_2(t,n,k)&:=n\sum_{j=t}^{\lfloor n/k\rfloor}\mu(\{X_0>u_n\}\cap
\{X_j>u_n\})\\
& \leq n\left\lfloor \frac
nk\right\rfloor\mu(X_0>u_n)^2+n\left\lfloor \frac
nk\right\rfloor(g^{-1}(u_n))^{\theta(d+\eta)} +n\left\lfloor \frac
nk\right\rfloor (g^{-1}(u_n))^{-1-\eta} C\e^{-\alpha t}.
\end{align*}
Recalling that $\mu(X_0>u_n)\sim \tau n^{-1}$ and
$g^{-1}(u_n)\sim\left(\frac{\tau}{\kappa\rho(\zeta)n}\right)^{1/d},$
we have
\[
S_2(t,n,k)=O\left(\frac1k+\frac{n^2}k
n^{-\theta(d+\eta)/d}+\frac{n^2}k n^{(1+\eta)/d}\e^{-\alpha'
\log^2(n)}\right).
\]
So, if we chose $\eta$ so that $\theta(d+\eta)/d>2$ then
$\lim_{k\to\infty}\limsup_{n\to\infty} S_2(t,n,k)=0$.

We are left with the first period from 1 to $t$ and the respective
sum
\[
S_1(t,n):=n\sum_{j=1}^{t}\mu(\{X_0>u_n\}\cap \{X_j>u_n\}).
\]
We set $\v=\v(n)=\lfloor(3g^{-1}(u_n))^{-1/\psi}\rfloor$, where
$\psi$ is given in Lemma~\ref{lem:bad-set}. Observe that
\[
\{X_0>u_n\}=B_{g^{-1}(u_n)}(\zeta)\subset B_{\v^{-\psi}}(\zeta)
\]
and, if $y\in\{X_0>u_n\}\cap \{X_j>u_n\}$, then
\[
\dist(f^j(y),y)\leq \dist(f^j(y),\zeta)+\dist(\zeta,y)\leq
2g^{-1}(u_n)<\v^{-\psi},
\]
which implies that
\begin{equation}
\label{eq:auxliary-computation2} \{X_0>u_n\}\cap \{X_j>u_n\}\subset
B_{\v^{-\psi}}(\zeta)\cap E_{\v^\psi}.
\end{equation}
We take $n$ so large that $\v=\v(n)\ge \v_0(\zeta)$. Hence
$\zeta\notin F_\v$. Using \eqref{eq:auxliary-computation2}, the
definition of $F_\v$ and the fact
$g^{-1}(u_n)\sim\left(\frac{\tau}{\kappa\rho(\zeta)n}\right)^{1/d},$
we have
\[
\mu(\{X_0>u_n\}\cap \{X_j>u_n\})=O(\v^{-\psi(d+\rho)}) =O(n^{-(d+\rho)/d}).
\]
Hence, $\limsup_{n\to\infty} S_1(t,n)\leq\limsup_{n\to\infty}
O\left(n\log^2(n)n^{-(d+\rho)/d}\right)=0$.
\end{proof}

\subsection{An Example}\label{subsec:example}

Here we present a $C^1$ open class of local diffeomorphisms with no
critical points that are non-uniformly expanding in the sense of
\cite{ABV,Alves}. Namely, let \( f: M \to M \) be a \( C^{1} \) local
diffeomorphism, we say that $f$ is non-uniformly expanding  if there exists a  constant \( \lambda
> 0 \)
 such that for Lebesgue almost all points \( x\in M \) the following
 \emph{non-uniform expansivity} condition is satisfied:
\begin{equation}\label{eq:non-unif-expand}
    \liminf_{n\to\infty}\frac{1}{n}\sum_{i=0}^{n-1}
    \log \|Df_{f^{i}(x)}^{-1}\|^{-1}\geq \lambda >0.
\end{equation} Condition \eqref{eq:non-unif-expand} implies that the \emph{expansion
time} function
\[
\mathcal E(x) = \min\left\{N: \frac{1}{n} \sum_{i=0}^{n-1} \log
\|Df^{-1}_{f^{i}(x)}\|^{-1} \geq \lambda/2 \ \ \forall n\geq
N\right\}
\]
is defined and finite almost everywhere in \( M \). We think of this
as the waiting time before the exponential derivative growth
kicks in.
We are now able to define the \emph{Hyperbolic tail set}, at time
$n\in\N$,
\begin{equation}
  \label{eq:tail_set}
  \Gamma_n=\left\{x\in I:\;
\mathcal{E}(x)>n\right\},
\end{equation}
which can be seen as the set of points that at time $n$ have not
reached a satisfactory exponential growth of the derivative. Applying \cite{ALP} and \cite{Gou2} together shows that these maps admit a Young tower whose return time tail is related to the volume decay rate of the hyperbolic tail set.

The class we consider here is obtained by deformation of a uniformly
expanding map by isotopy inside some small region. In general, these
maps are not expanding: deformation can be made in such way that the
new map has periodic saddles. We follow the construction in
\cite{ABV,Alves}.

Let $M$ be any compact Riemannian $d$-dimensional manifold supporting some uniformly expanding
map $f_0$: there exists $\sigma_0>1$ such that
$\|Df_0(x)v\|>\sigma_0\|v\|$ for every $x\in M$ and every $v\in T_x
M$. Let $V\subset M$ be small compact domain, so that $f_0|_V$ is
one-to-one. Let $f_1$ be a $C^1$ map coinciding with $f_0$
in $M \backslash V$ for which the following holds:

\begin{enumerate}

    \item $f_1$ is \textit{volume expanding everywhere}:
    there is $\sigma_1>1$ such that
    \[
    |\det Df_1(x)|>\sigma_1,\;\;\mbox{ for every } x\in M;
    \]

    \item $f_1$ is \textit{not too contracting on} $V$: there is
    small $\delta>0$ such that
    \[
    \|Df_1(x)^{-1}\|<1+\delta,\;\;\mbox{ for every } x\in V.
    \]

\end{enumerate}
We consider the class of maps $f$ in a small $C^1$-neighbourhood $\mathcal F$ of $f_1$.

In \cite[Section~6]{Alves} it was shown that these maps satisfy condition
\eqref{eq:non-unif-expand} and there exist $C,\gamma>0$ such that
$\l(\Gamma_n)\leq C\e^{-\gamma n}$ for all $n\in\N$. Now, using the results in
\cite{Gou2} this implies that every map $f\in\mathcal F$ admits a Young tower for which conditions
\eqref{Y-Markov}--\eqref{Y-exponential-tail} are satisfied. This means that we can apply
Theorem~\ref{thm:Collet-multi-dimensional} and obtain
that all assertions of Corollary~\ref{cor:HTS-Collet-higher-dim} hold for this
class of maps $\mathcal F$.


\begin{thebibliography}{BDV05}
\providecommand{\bibinfo}[2]{#2}

\bibitem[Ab]{Abadi} M.\ Abadi,
{\em Sharp error terms and necessary conditions for exponential
hitting times in mixing processes,} Ann. Probab. {\bf 32} (2004)
243--264.

\bibitem[AG]{AbGal} M.\ Abadi, A.\ Galves,
{\em Inequalities for the occurrence times of rare events in mixing processes. The state of the art,}
%Inhomogeneous random systems (Cergy-Pontoise, 2000).
Markov Process. Related Fields  {\bf 7}  (2001) 97--112.

\bibitem[Al]{Alves} J.F.\ Alves,
{\em Strong statistical stability of non-uniformly expanding maps,}
Nonlinearity {\bf 17} (2004) 1193–-1215.

\bibitem[ABV]{ABV} J.F.\ Alves, C.\ Bonatti, M.\ Viana,
{\em SRB measures for partially hyperbolic systems whose central
direction is mostly expanding,}  Invent. Math.  {\bf 140} (2000)
351–-398.

\bibitem[ALP]{ALP} J.F.\ Alves, S.\ Luzzatto, V.\ Pinheiro,
{\em    Markov structures and decay of correlations for non-uniformly expanding dynamical systems,}  Ann. Inst. H. Poincar\'{e}
Anal. Non Lin\'eaire  {\bf 22} (2005) 817-839.

\bibitem[BC]{BC85}
\bibinfo{author}{M.~Benedicks} and \bibinfo{author}{L.~Carleson},
  \emph{\bibinfo{title}{On iterations of $1-ax^2$ on $(-1,1)$}},
  \bibinfo{journal}{Ann. Math.} \textbf{\bibinfo{volume}{122}}
  (\bibinfo{year}{1985}), \bibinfo{pages}{1--25}.

\bibitem[BRSS]{BRSS} H.\ Bruin, J.\ Rivera-Letelier, W.\ Shen,
S.\ van Strien,
{\em Large derivatives, backward contraction and invariant densities
for interval maps,} Invent. Math. {\bf 172} (2008) 509--593.

\bibitem[BSTV]{BSTV} H.\ Bruin, B.\ Saussol, S.\ Troubetzkoy,
S.\ Vaienti, {\em Return time statistics via inducing,} Ergodic
Theory Dynam. Systems {\bf 23}  (2003) 991--1013.

\bibitem[BT]{BTintstat} H.\ Bruin, M.\ Todd,
{\em Return time statistics for invariant measures for interval maps
with positive Lyapunov exponent,}
arXiv:0708.0379.

\bibitem[BV]{BrV} H.\ Bruin, S.\ Vaienti,
{\em S. Return time statistics for unimodal maps,} Fund. Math.
{\bf 176} (2003) 77--94.

\bibitem[C1]{Coelho} Z.\ Coelho,
{\em Asymptotic laws for symbolic dynamical systems,}
Topics in symbolic dynamics and applications (Temuco 1997) 123--165,
LMS Lecture Notes Series {\bf 279} Cambridge Univ. Press, 2000.

\bibitem[C2]{Coe} Z.\ Coelho,
{\em The loss of tightness of time distributions for homeomorphisms
of the circle,}
Trans. Amer. Math. Soc. {\bf 356}  (2004) 4427--4445

\bibitem[CF]{CoedeF} Z.\ Coelho, E.\ de Faria, {\em Limit laws of entrance times for homeomorphisms of the circle,}  Israel J. Math.  {\bf 93} (1996) 93--112.

\bibitem[Col1]{Collint} P.\ Collet,
{\em Some ergodic properties of maps of the interval,}
Dynamical Systems (Temuco, 1991/1992),
(Travaux en cours, 52). Herman, Paris, 1996, pp. 55–91

\bibitem[Col2]{Collexval} P.\ Collet,
{\em Statistics of closest return for some non-uniformly hyperbolic
systems,}
Ergodic Theory Dynam. Systems {\bf 21}  (2001) 401--420.

\bibitem[DGS]{DenGoSha} M.\ Denker, M.\ Gordin, A.A.\ Sharova,
{\em Poisson limit theorem for toral automorphisms,}
Illinois J. Math. {\bf 48}  (2004) 1--20.

\bibitem[Do]{Dol} D.\ Dolgopyat,
 {\em Limit theorems for partially hyperbolic systems,} Trans. Amer. Math. Soc. {\bf 356}  (2004) 1637--1689

\bibitem[DM]{DM} E.J.\ Dudewicz, S.N.\ Mishra, \emph{Modern Mathematical Statistics,} Wiley
Series in Probability and Mathematical Statistics, John Wiley \& Sons, 1988.

\bibitem[Fe]{Fe52}
W.~Feller, \emph{An introduction to
  {P}robability {T}heory and its {A}pplications}, Volume I.
  Wiley Publications in Statistics, 1952.

\bibitem[FF1]{FF}  A.C.M.\ Freitas, J.M.\ Freitas,
{\em Extreme values for Benedicks Carleson maps,} To appear in
Ergodic Theory Dynam. Systems.

\bibitem[FF2]{FF2}  A.C.M.\ Freitas, J.M.\ Freitas,
{\em On the link between dependence and independence
 in Extreme Value Theory for Dynamical Systems,}
Stat. Probab. Lett. {\bf 78} (2008) 1088--1093.

\bibitem[G1]{Gou1} S.\ Gou\"ezel,
{\em Sharp polynomial estimates for the decay of correlations,}
Israel J. Math. {\bf 139}  (2004) 29--65.

\bibitem[G2]{Gou2} S.\ Gou\"ezel,
{\em Decay of correlations for nonuniformly expanding systems,}
Bull. Soc. Math. France {\bf 134} (2006) 1--31.

\bibitem[HLV]{HLV}N.\ Haydn, Y.\ Lacroix, S.\ Vaienti,
{\em Hitting and return times in ergodic dynamical systems,}
Ann. Probab. {\bf 33}  (2005) 2043--2050.

\bibitem[H]{Hir} M.\ Hirata,
{\em Poisson law for Axiom A diffeomorphisms,}
Ergodic Theory Dynam. Systems {\bf 13} (1993) 533--556.

\bibitem[HSV]{HSV} M.\ Hirata, B.\ Saussol, S.\ Vaienti,
{\em Statistics of return times: a general framework and new
applications,}
Comm. Math. Phys. {\bf 206}  (1999) 33--55.

\bibitem[HHL]{HHL}  T.\ Hsing,  J.\ H\"usler, M.R.\ Leadbetter, \emph{On the exceedance point process for a stationary sequence,}
Probab. Theory Related Fields \textbf{78} (1988) 97--112.


\bibitem[Ka]{Ka83} O.~Kallenberg, \emph{Random measures}, Academic
Press Inc., New York, 1986.

\bibitem[KN]{KN92}
G.~Keller and T.~Nowicki, \emph{Spectral theory, zeta functions and
the distribution of periodic points for {C}ollet-{E}ckmann maps},
Comm. Math. Phys. \textbf{149} (1992) 31--69.

\bibitem[LLR]{LLR83}
G.~Lindgren, M.R.~Leadbetter, and H.~Rootz\'en, \emph{Extremes and
related properties of random sequences and processes}, Springer
Series in Statistics, Springer-Verlag, New York-Berlin, 1983, XII.

\bibitem[Pi]{Pit} B.\ Pitskel,
{\em Poisson limit law for Markov chains},
Ergodic Theory Dynam. Systems {\bf 11} (1991) 501--513.

\bibitem[R]{Rych} M.\ Rychlik,
{\em Bounded variation and invariant measures,}
Studia Math. {\bf 76} (1983) 69--80.

\bibitem[Y1]{Yo92}
L.~S. Young, \emph{Decay of correlations for certain quadratic
maps},
Comm. Math. Phys. \textbf{146} (1992) 123--138.

\bibitem[Y2]{Yostat}  L.S.\ Young,
{\em Statistical properties of dynamical systems with some hyperbolicity,}
Ann. of Math. (2) {\bf 147}  (1998) 585--650.

\bibitem[Y3]{Yorrr}L.S.\ Young,
{\em  Recurrence times and rates of mixing,}
Israel J. Math. {\bf 110}  (1999) 153--188.

\end{thebibliography}
\end{document}